\documentclass[12pt, a4paper]{amsart}

\usepackage{amsmath, amssymb, amsfonts, amsthm}
\usepackage{mathtools}
\usepackage{latexsym}
\usepackage{enumitem}
\usepackage[super]{nth}

\usepackage[left=30mm, right=30mm, top=20mm, bottom=20mm]{geometry}
\usepackage{tikz-cd}

\usepackage{hyperref}
\usepackage[colorinlistoftodos,textsize=footnotesize,textwidth=15ex]{todonotes}

\newtheorem{thmintro}{Theorem}

\newtheorem{corintro}[thmintro]{Corollary}

\theoremstyle{definition}
\newtheorem{rkintro}{Remark}

\newtheorem{expintro}{Example}

\theoremstyle{plain}
\newtheorem{theorem}{Theorem}[section]
\newtheorem{proposition}[theorem]{Proposition}
\newtheorem{lemma}[theorem]{Lemma}
\newtheorem{corollary}[theorem]{Corollary}

\theoremstyle{definition}
\newtheorem{definition}[theorem]{Definition}
\newtheorem{example}[theorem]{Example}
\newtheorem{remark}[theorem]{Remark}

\renewcommand{\phi}{\varphi}
\renewcommand{\epsilon}{\varepsilon}
\newcommand{\NN}{\mathbb{N}}
\newcommand{\ZZ}{\mathbb{Z}}
\newcommand{\proj}{\mathrm{proj}}
\newcommand{\inv}{^{-1}}
\newcommand{\supp}{\mathrm{supp}}
\newcommand{\Pc}{\mathrm{Pc}}
\newcommand{\Aut}{\mathrm{Aut}}
\newcommand{\Fix}{\mathrm{Fix}}
\newcommand{\Stab}{\mathrm{Stab}}
\newcommand{\con}{\mathrm{con}}
\newcommand{\nub}{\mathrm{nub}}
\newcommand{\mc}[1]{\mathcal{#1}}
\newcommand{\la}{\langle}
\newcommand{\ra}{\rangle}

\DeclareMathOperator{\conv}{conv}
\DeclareMathOperator{\dist}{d_{\mathrm{Ch}}}
\DeclareMathOperator{\esupp}{esupp}
\DeclareMathOperator{\geo}{geo+}
\DeclareMathOperator{\Ad}{Ad}
\DeclareMathOperator{\dcat}{d}

\DeclareMathOperator{\ma}{ma}
\DeclareMathOperator{\re}{re}
\DeclareMathOperator{\im}{im}
\DeclareMathOperator{\imp}{im+}
\DeclareMathOperator{\rep}{re+}
\DeclareMathOperator{\map}{ma+}
\DeclareMathOperator{\pma}{pma}

\newcommand{\g}{\mathfrak{g}}
\DeclareMathOperator{\height}{ht}
\DeclareMathOperator{\GL}{GL}
\newcommand{\co}{\colon\thinspace}
\DeclareMathOperator{\sph}{sph}
\newcommand{\hh}{\mathfrak{h}}
\newcommand{\T}{\mathfrak{T}}
\newcommand{\U}{\mathfrak{U}}

\newcommand{\G}{\mathfrak{G}}

\newcommand{\n}{\mathfrak{n}}

\newcommand{\Jaff}{J_{\mathrm{aff}}}
\newcommand{\JaffCox}{J_{\mathrm{aff}}^{\mathrm{Cox}}}
\newcommand{\Jsph}{J_{\mathrm{sph}}}
\newcommand{\Jperp}{J^{\perp}}

\numberwithin{equation}{section} 

\begin{document}

\title{Describing the nub in maximal Kac--Moody groups}

\dedicatory{Dedicated to Bernhard Mühlherr on the occasion of his $\nth{60}$ birthday}

\author[S.~Bischof]{Sebastian \textsc{Bischof}}
\address{Universit\"at Paderborn, Institut f\"ur Mathematik, 33098 Paderborn, Germany}
\email{sebastian.bischof@math.uni-paderborn.de}

\author[T.~Marquis]{Timoth\'ee \textsc{Marquis}}
\address{UCLouvain, IRMP, 1348 Louvain-la-Neuve, Belgium}
\email{timothee.marquis@uclouvain.be}

\thanks{SB is supported by a fellowship of the DAAD (57664192) and by the Deutsche Forschungsgemeinschaft (DFG, German Research Foundation) -- 536272274 (Walter Benjamin Programme). TM is a F.R.S.-FNRS Research associate, and is supported in part by the FWO and the F.R.S.-FNRS under the EOS programme (project ID 40007542).}



\begin{abstract}
	Let $G$ be a totally disconnected locally compact (tdlc) group. The contraction group $\mathrm{con}(g)$ of an element $g\in G$ is the set of all $h\in G$ such that $g^n h g^{-n} \to 1_G$ as $n \to \infty$. The nub of $g$ can then be characterized as the intersection $\mathrm{nub}(g)$ of the closures of $\mathrm{con}(g)$ and $\mathrm{con}(g^{-1})$.
	
	Contraction groups and nubs provide important tools in the study of the structure of tdlc groups, as already evidenced in the work of G.~Willis. It is known that $\mathrm{nub}(g) = \{1\}$ if and only if $\mathrm{con}(g)$ is closed. In general, contraction groups are not closed and computing the nub is typically a challenging problem.
	
	Maximal Kac--Moody groups over finite fields form a prominent family of non-discrete compactly generated simple tdlc groups. In this paper we give a complete description of the nub of any element in these groups.
\end{abstract}

\maketitle

\section{Introduction}

One of today's main approaches to study totally disconnected locally compact (tdlc) groups is \emph{scale theory}, initiated by G.~Willis in \cite{Wi94} and \cite{Wi01}. It makes essential use of the classical result of D.~van~Dantzig that any tdlc group has a basis of identity neighborhoods consisting of compact open subgroups. Willis introduced the \emph{scale} $s_G(\alpha)$ of an automorphism $\alpha$ of  a tdlc group $G$ as
\[ s_G (\alpha) = \min\{ [ \alpha(U) : \alpha(U) \cap U] \mid U\leq G \text{ compact and open} \}. \]
A compact open subgroup $U\leq G$ is called \emph{minimizing for $\alpha \in \Aut(G)$} if $s_G(\alpha) = [\alpha(U) : \alpha(U) \cap U]$. In general it is a challenging problem to compute scales or to describe minimizing subgroups. The scale and a minimizing subgroup were computed for automorphisms of $p$-adic Lie groups by H.~Gl\"{o}ckner in \cite{Glockner98} and for inner automorphisms of closed Weyl-transitive subgroups of the automorphism group of a locally finite building by Baumgartner, Parkinson and Ramagge in \cite{BPR19}.

In \cite{BW04}, Baumgartner and Willis introduced for an automorphism $\alpha$ the \emph{nub of $\alpha$}, denoted by $\nub(\alpha)$, as the intersection of all subgroups which are minimizing for $\alpha$. They showed that the nub of $\alpha$ is closely related to its \emph{contraction group} $\con(\alpha) = \{ g\in G \mid \lim_{n\to \infty} \alpha^n(g) = 1 \}$. More precisely, they showed that $\overline{\con(\alpha)} = \con(\alpha) \nub(\alpha)$ and that $\con(\alpha)$ is closed if and only if $\nub(\alpha) = \{1\}$. For instance, it is known that linear groups over local fields have closed contraction groups, hence trivial nubs. Thanks to decomposition results by Glöckner and Willis (\cite{GW10}, \cite{Glockner-Willis21:Decomposition_of_locally_compact_contraction_groups} and \cite{GW21}) closed contraction groups are quite well understood. In particular, having closed contraction groups is a restrictive condition (see also \cite{Caprace_DeMedts2013}).

In general, contraction groups are not closed and it would be interesting to have a concrete description of the nub in terms of a (topological) generating set or, if $G$ acts on some geometric structure, to have a geometric interpretation of the nub. This is for instance motivated by another result of Willis (see \cite[Corollary~$4.1$]{Wi14}) which provides an algorithm to compute the scale and a minimizing subgroup of an automorphism once a concrete description of the nub is known. In particular, obtaining a concrete description of the nub is more difficult than computing the scale or a minimizing subgroup. As we have already mentioned above both are typically difficult problems.

For `well-behaved' compact groups $G$ the nub is known (see \cite{Wi14}). In the non-compact case, to the best of our knowledge, the nub of an automorphism $\alpha$ with $s_G(\alpha) >1$ (this is to avoid uninteresting examples like discrete groups) has only been computed explicitly for the automorphism group of a regular tree, for groups acting on trees satisfying Tits' independence property (P) and variations thereof. But in these cases the computation is straightforward. Let us also mention the recent paper \cite{Reid_24scalefunctionlocallycompact} in which C.~Reid characterizes the set of fixed points of the nub of an automorphism in certain groups acting by isometries on a non-positively curved space.

A prominent family of non-discrete compactly generated simple tdlc groups are maximal Kac--Moody groups obtained as completions of minimal Kac--Moody groups $\mathfrak{G}_A(k)$ (as introduced in \cite{Ti87}) over finite fields. There are a few different ways to construct such completions. The most natural one from an algebraic perspective, which avoids degeneracy issues in small characteristic, is the scheme-theoretic completion $G^{\pma} = \G^{\pma}_A(k)$ where $A = (a_{ij})_{i, j \in I}$ is a generalized Cartan matrix and $k = \mathbb{F}_q$ is a finite field (see \cite[\S 8.5]{Ma18}). It comes with a canonical BN-pair $(B, N)$, with associated Weyl group $W \cong N / (B \cap N)$. We fix a section $$W \to N: w\mapsto\widetilde{w}$$ of the quotient map $N \to W$.

Contraction groups in maximal Kac--Moody groups have been studied in \cite{BRR08}. The authors of \emph{loc.\ cit.} showed that when $A$ is of indefinite type, for any element $g$ which is not topologically periodic there exists a conjugate of a root group $U_\alpha$ (see below) which is contracted by $g$ and $g\inv$. As a consequence they obtained that such contraction groups are not closed.

The main goal of this paper is to describe the nub of any element in $G^{\pma}$. It is well-known that topologically periodic elements have trivial contraction group (and hence trivial nub). Thus it suffices to consider elements $g\in G^{\pma}$ which are not topologically periodic. Computing the nub of such $g$ can be further reduced to the case where $g$ is an element of the Weyl group $W$ of $G^{\pma}$ (or rather one of its preimages under the canonical homomorphism $N \to W$).

More precisely, let $S = \{ s_i \mid i \in I \}$ denote the canonical generating set of $W$ (which we will identify with $I$ in the sequel). Let $w\in W$ and let $J = \supp(w)$ be the support of $w$, that is, $J$ is the smallest subset of $S$ with $w\in W_J := \la J \ra$. Let also $\Jaff \subseteq J$ be the union of all affine components of the Dynkin diagram of $J$, and define $\Jperp = \{ s\in S \setminus J \mid \forall j\in J: sj = js \}$. For the computation of the nub we can further assume (see Proposition~\ref{prop: The nub only for straight and standard elements}) that $w$ is straight and standard --- see \S\ref{subsection:CoxSyst} and \S\ref{Subsection: Standard elements} for precise definitions. This implies that there exists a unique maximal spherical subset $\Jsph^w$ of $J \setminus \Jaff$ such that $w$ normalizes $W_{\Jsph^w}$.

As in the classical setting of Cartan matrices, there is a root system $\Delta = \Delta(A)$ associated to $A$ on which $W$ acts. Let $Q = \bigoplus_{i\in I} \ZZ \alpha_i$ be the associated root lattice, where $\Pi = \{ \alpha_i \mid i \in I \}$ is the set of simple roots. Let $\Delta^{\pm} \subseteq Q$ be the set of positive/negative roots so that $\Delta = \Delta^+ \dot{\cup} \Delta^-$. Let also $\Phi = W.\Pi$ be the set of real roots. For $\alpha = \sum_{i\in I} n_i \alpha_i \in Q$ set $\supp(\alpha) := \{ i\in I \mid n_i \neq 0 \}$. For $J \subseteq I$ we let $\Delta^+(J) := \{ \alpha \in \Delta^+ \mid \supp(\alpha) \subseteq J \}$. We further set
$$ \Delta_{w\pm} := \{ \alpha \in \Delta^+ \mid \exists n\in \NN: w^{\pm n} \alpha \in \Delta^- \}. $$
A subset $\Psi \subseteq \Delta$ is called \emph{closed} if $\alpha + \beta \in \Psi$ whenever $\alpha, \beta \in \Psi$ and $\alpha + \beta \in \Delta$. For a closed subset $\Psi \subseteq \Delta^+$, we set $U^{\ma}_{\Psi} := \overline{ \la U_{\alpha} \mid \alpha \in \Psi \ra }$ where $U_{\alpha} \leq G^{\pma}$ denotes the root group associated to the root $\alpha$ (see \S\ref{Section: Maximal Kac--Moody groups}).

\begin{thmintro}\label{Thmintro: Main result}
	Let $w\in W$ be straight and standard and set $J = \supp(w)$. Then
	$$\nub(\widetilde{w}) = U^{\ma}_{K(w)} \quad \text{where} \quad K(w) = \Delta^+ \setminus \left( \Delta_{w+} \cup \Delta_{w-} \cup \Delta^+(J^{\perp} \cup J_{\mathrm{aff}} \cup \Jsph^w) \right). $$
	In general, if $g \in G^{\pma}$ is not topologically periodic, there exists $w\in W$ straight and standard and $h\in G^{\pma}$ with $\nub(g) = h \nub(\widetilde{w}) h\inv$.
\end{thmintro}

The proof of Theorem~\ref{Thmintro: Main result} is given at the end of \S\ref{subsection:UBFTNKM} (see Proposition~\ref{prop: The nub only for straight and standard elements} for the proof of the second statement).

\begin{expintro}\label{Expintro}
	Suppose $(W, S)$ is of indefinite type and let $w\in W$ be straight such that $\supp(w) = S$. Assume, moreover, that $w$ does not normalize any spherical parabolic subgroup (for instance, if $w$ is a Coxeter element, see \cite[Theorem~C(2)]{Ma23}). Then $K(w) = \Delta^+ \setminus \left( \Delta_{w+} \cup \Delta_{w-} \right)$. As we will see in Lemma~\ref{lemma:FixLDeltawpm}, Theorem~\ref{Thmintro: Main result} then amounts to $\nub(\widetilde{w}) = \Fix(L_w) \cap U^{\map}$, where $L_w=\{w^z \ | \ z\in\ZZ\}$ is viewed as a subset of the fundamental apartment in the building associated to the BN-pair $(B,N)$ (see \S\ref{subsection:RGDsystems}). \\	
	Similarly, if $A = \begin{psmallmatrix}
		2 & a_{12} \\ a_{21} & 2
	\end{psmallmatrix}$ with $a_{12} a_{21} >4$, then for any $w\in W$ of infinite order we have $\nub(\widetilde{w}) = U^{\ma}_{\Delta^{\imp}} = \Fix(L_w) \cap U^{\map}$, where $\Delta^{\imp} = \Delta^+ \setminus \Phi$.
\end{expintro}

\begin{rkintro}
	The pair $(\mathfrak{G}_A(k), (U_{\alpha})_{\alpha \in \Phi})$ is an \emph{RGD system}. In the present paper we establish upper and lower bounds for the nub in arbitrary completions of RGD systems of non-spherical type with finite root groups. For more details we refer to Theorem~\ref{Theorem: Upper bound for nub} and Proposition~\ref{prop: root groups in the nub}. In some cases our upper bound is already optimal by Theorem~\ref{Thmintro: Main result}.
\end{rkintro}

\begin{rkintro}
	There exist various completions of $\mathfrak{G}_A(k)$ in the literature, obtained as quotients of $G^{\pma}$ by a compact normal subgroup. Note that Theorem~\ref{Thmintro: Main result} also provides a description of the nub of an element in these completions, see Remark~\ref{Remark: nub in quotients}.
\end{rkintro}

As first consequences of Theorem~\ref{Thmintro: Main result} we can characterize when an element has trivial nub and when two elements have the same nub.

\begin{corintro}\label{Corintro: trivial nub}
	Let $w\in W$ be straight, and let $J = \supp(w)$. Then $\nub(\widetilde{w}) = \{1\}$ if and only if $J$ is a union of affine components of $I$.
\end{corintro}

If $w\in W_J$ and $J'$ is a component of $J \subseteq I$, we write $w_{J'}$ for the image of $w$ under the projection $W_J \to W_{J'}$. For $w\in W$ set $L_w = \{ w^z \mid z\in \ZZ \}$ and denote by $\conv(L_w)$ the convex hull of $L_w$ in the Coxeter complex of $(W, S)$ (see Example~\ref{Example: Coxeter building}).

\begin{corintro}\label{Corintro: Same nub}
	Let $w,v\in W$ be straight and standard, and let $J = \supp(w)$ and $K = \supp(v)$. Then the following are equivalent:
	\begin{enumerate}
		\item 
		$\nub(\widetilde{w}) = \nub(\widetilde{v})$;
		\item
		$\Jperp \cup \Jaff = K^{\perp} \cup K_{\mathrm{aff}}$ and $J\setminus\Jaff=K\setminus K_{\mathrm{aff}}$ and $\conv(L_{w_{J'}}) = \conv(L_{v_{J'}})$ for every component $J'$ of $J\setminus\Jaff$ with $|J'|\geq 3$.
	\end{enumerate}
\end{corintro}

\begin{expintro}
Here are two examples to illustrate Corollary~\ref{Corintro: Same nub}.
\begin{enumerate}
\item
Suppose that (the Dynkin diagram of) $I$ is of the form $I=I'\sqcup\{j\}$ where $I'$ has two components $J,K$ of affine type and both $J\cup\{j\}$ and $K\cup\{j\}$ are connected. Let $w,v\in W$ be straight with $\supp(w)=J$ and $\supp(v)=K$. Then $J=\Jaff$ and $K=K_{\mathrm{aff}}$ are disjoint, but $\Jperp \cup \Jaff = K^{\perp} \cup K_{\mathrm{aff}}$, so $\nub(\widetilde{w}) = \nub(\widetilde{v})$. Note that these nubs are non-trivial by Corollary~\ref{Corintro: trivial nub}.
\item
Suppose that (the Dynkin diagram of) $I$ has no components of affine type, and let $w,v\in W$ be straight and such that $I = \supp(w) = \supp(v)$ (thus, $I^\perp\cup I_{\mathrm{aff}}=\varnothing$). Assume moreover that for each component $J'$ of $I$ with $|J'|\geq 3$, there exist nonzero $r,s\in\ZZ$ such that $w_{J'}^r=v_{J'}^s$. Then $\conv(L_{w_{J'}}) = \conv(L_{v_{J'}})$ for each such $J'$. Moreover, up to conjugation, we may assume that $w,v$ are straight and standard (see \S\ref{Subsection: Standard elements}). Hence $\nub(\widetilde{w}) = \nub(\widetilde{v})$.
\end{enumerate}
\end{expintro}

Theorem~\ref{Thmintro: Main result} also allows us to compute explicitly closures of contraction groups in $G^{\pma}$. In particular, we obtain the following two corollaries (see \S\ref{subsection:proofCorBC}). 

\begin{corintro}\label{corintro:nuboverline}
	Let $g\in G^{\pma}$. Then $\nub(g) = \overline{\con(g) \cap \con(g\inv)}$.
\end{corintro}

Note that we always have $\nub(g) = \overline{\con(g)} \cap \overline{\con(g\inv)}$ (see \S\ref{subsection:TDLCG}), whereas the equality in Corollary~\ref{corintro:nuboverline} does not always hold (see \cite[Remark~3.1(d)]{Wi14}).

\begin{corintro}\label{corintro:congbar}
	Let $w\in W$ be straight and standard. Then
	\[ \overline{\con(\widetilde{w})} = U^-_{w+}\cdot U^{\ma}_{K'(w)} = \overline{\langle U_{\alpha} \mid U_{\alpha} \leq \con(\widetilde{w}) \rangle} \]
	where $U^-_{w+}:=\langle U_{\alpha} \ | \ \alpha\in -\Delta_{w+}\rangle$ and $K'(w):=K(w)\cup\Delta_{w-}$.
\end{corintro}

Finally, note that for any tdlc group $G$ and $g\in G$, denoting by $\gamma_g\co G \to G$ the conjugation by $g$, the pair $(\nub(g), \gamma_g)$ is an example of a ``dynamical system of algebraic origin'' (see e.g.\ \cite{Sc95}), i.e.\ a compact group and a an automorphism acting ergodically on it (see \cite[Theorem~4.1]{Wi14}). In \cite[Section~5]{Wi14} Willis discusses the ``finite depth condition", under which the nub has a finite decomposition series (for a compact tdlc group $G$ and $\alpha \in \Aut(G)$, the pair $(G, \alpha)$ has finite depth if there exists an open subgroup $V \leq G$ such that $\bigcap_{z\in \ZZ} \alpha^z(V) = \{1\}$). For Kac--Moody groups Theorem~\ref{Thmintro: Main result} implies that the nub is either trivial or has infinite depth.

\begin{corintro}\label{Corintro: infinite depth}
	Let $w\in W$ be straight with $\nub(\widetilde{w}) \neq \{1\}$. Then $(\nub(\widetilde{w}), \gamma_{\widetilde{w}})$ does not have finite depth.
\end{corintro}

\subsection*{Overview}

In Section~\ref{Section: Preliminaries} we fix notation and recall known results about tdlc groups, Coxeter systems, buildings and RGD systems. In Section~\ref{section:completionsRGD} we introduce the notions of completions of RGD systems and of unipotent subgroups to treat different completions of groups of Kac--Moody type in a uniform way. In Section~\ref{Section: Reduction steps} we reduce the computation of the nub of an arbitrary element to that of a straight and standard element $w\in W$. The goal of Section~\ref{Section: The nub in completions} is on the one hand to show that compact open unipotent subgroups are minimizing for $w$ and on the other hand to provide general upper and lower bounds for the nub in general completions of RGD systems. In Section~\ref{Section: The nub in Kac--Moody groups} we refine these upper and lower bounds for the scheme-theoretic completion of Kac--Moody groups and show that they coincide, thus proving Theorem~\ref{Thmintro: Main result}. We then deduce the Corollaries~\ref{Corintro: trivial nub}, \ref{Corintro: Same nub}, \ref{corintro:nuboverline}, \ref{corintro:congbar} and \ref{Corintro: infinite depth}.

\subsection*{Acknowledgement}

The first author thanks Colin Reid and George Willis for helpful discussions about the topic. We would also like to thank Pierre-Emmanuel Caprace for his useful comments on an earlier version of the paper. Finally, we would like to thank the anonymous referee for their careful reading of the paper and detailed comments.

\section{Preliminaries}\label{Section: Preliminaries}
For basics on totally disconnected locally compact groups, we refer to the reader to e.g.\ \cite{Wi01}, \cite{Wi14}. A standard reference on Coxeter groups, buildings and RGD systems is \cite{AB08}.

\subsection{Totally disconnected locally compact (tdlc) groups}\label{subsection:TDLCG}

Let $G$ be a tdlc group. It is a classical result by Van~Dantzig that the set of all compact open subgroups of $G$ forms a basis of neighborhoods of the identity. The \emph{scale function} $s$ of $G$ is defined as follows:
\[ s: G \to \NN, \, g \mapsto \min \{ [ gUg\inv : gUg\inv \cap U ] \mid U\leq G \text{ compact and open} \}. \]
Let $g \in G$ and let $U \leq G$ be a compact open subgroup. Then $U$ is said to be \emph{minimizing for $g$} if $s(g) = [ gUg\inv : gUg\inv \cap U ]$. We say that $U$ is \emph{tidy for $g$} if the following two conditions are satisfied:
\begin{enumerate}[label=(T\Alph*)]
	\item $U = U_{g +} U_{g -}$, where $U_{g+} := \bigcap_{n\geq 0} g^n U g^{-n}$ and $U_{g-} := \bigcap_{n\geq 0} g^{-n} U g^n$;
	
	\item $\widehat{U}_{g +} := \bigcup_{n\geq 0} g^n U_{g +} g^{-n}$ and $\widehat{U}_{g -} := \bigcup_{n\geq 0} g^{-n} U_{g -} g^n$ are closed in $G$.
\end{enumerate}

A compact open subgroup is said to be \emph{tidy above for $g$} it is satisfies (TA), and \emph{tidy below for $g$} if it satisfies (TB). The next theorem shows the connection between minimizing and tidy subgroups.

\begin{theorem}[{\cite[Theorem~3.1]{Wi01}}]\label{Theorem: Theorem 3.1 in Wi01}
	Let $g\in G$ and let $U \leq G$ be a compact open subgroup of $G$. Then $U$ is minimizing for $g$ if and only if $U$ is tidy for $g$.
\end{theorem}

The following results are useful in this paper.

\begin{lemma}\label{Lemma: auxiliary results for tidy subgroups}
	Let $g, h \in G$ and let $U \leq G$ be a compact open subgroup of $G$. 
	\begin{enumerate}
		\item $U$ is tidy for $g$ if and only if $[U : U \cap g^{-n} U g^n] = [U : U \cap g^{-1} U g]^n$ holds for all $n \in \NN$.
		
		\item $U$ is tidy for $g$ if and only if $U$ is tidy for $g\inv$.
		
		\item $U$ is tidy for $g$ if and only if $h U h\inv$ is tidy for $hgh\inv$. In particular, if $U$ is tidy for $g$, then $g^z U g^{-z}$ is tidy for $g$ for all $z\in \ZZ$.
	\end{enumerate}
\end{lemma}
\begin{proof}
	Part $(1)$ is \cite[Corollary~3.5]{Mo02}; Part $(2)$ is an immediate consequence of the definition of tidiness; Part $(3)$ can be checked using Theorem~\ref{Theorem: Theorem 3.1 in Wi01}.
\end{proof}

Let $g\in G$. The \emph{contraction group of $g$} is defined as follows:
\[ \con_G(g) := \{ h\in G \mid \lim_{n \to \infty} g^n h g^{-n} = 1_G \}. \]
It is well-known that for $g, h \in G$ and $n\in \NN^*$ we have $\con(g) = \con(g^n)$ and $h\con(g)h\inv = \con( hgh\inv )$. It follows from \cite[Remark~3.1(d)]{Wi14} that the intersection of all subgroups being tidy for $g$ coincides with the intersection of the closures of the contraction groups of $g$ and $g\inv$. This common subgroup is called the \emph{nub of $g$} and is denoted by
\[ \nub_G(g) := \bigcap \{ U \leq G \mid U \text{ is tidy for } g \} = \overline{\con_G(g)} \cap \overline{\con_G(g\inv)}. \]
We simply write $\con(g)$ and $\nub(g)$, if the group $G$ is clear from the context. For $g\in G$ we call a subset $X \subseteq G$ \emph{$g$-stable} if $gXg\inv = X$.

\begin{lemma}[{\cite[Corollary~$4.3$]{Wi14}}]\label{Lemma: Facts about nub}
	Let $g\in G$. Then $\nub(g)$ does not have any proper, relatively open, $g$-stable subgroups.
\end{lemma}

\begin{remark}\label{Remark: nub in quotients}
	Let $G$ be a tdlc group and $N$ be a compact normal subgroup of $G$. Let $\phi: G \to G/N$ be the canonical projection and let $g\in G$. Then it follows from \cite[Proposition~4.11]{Re16} that $\phi(\nub_G(g)) = \nub_{G/N}(\phi(g))$ (note that the nub of $g$ coincides with the nub of the flat group $\langle g\rangle$ in the terminology of \emph{loc. cit.}).
\end{remark}

\subsection{Coxeter systems}\label{subsection:CoxSyst}
Let $(W, S)$ be a Coxeter system, which we will always assume to be of finite rank, that is, with $S$ finite. We shall denote by $\ell=\ell_S$ the corresponding word metric with respect to $S$.

It is well-known that for each $J \subseteq S$ the pair $(W_J := \langle J \rangle, J)$ is a Coxeter system (cf.\ \cite[Ch.\ IV, §$1$ Theorem $2$]{Bo68}). For $s, t \in S$ we denote the order of $st$ in $W$ by $m_{st}$. The \textit{Coxeter diagram} corresponding to $(W, S)$ is the labeled graph $(S, E(S))$, where $E(S) = \{ \{s, t \} \mid m_{st}>2 \}$ and where each edge $\{s,t\}$ is labeled by $m_{st}$ for all $s, t \in S$. The Coxeter system $(W, S)$ is called \emph{irreducible} if the underlying Coxeter diagram is connected. We say that $J \subseteq S$ is \emph{irreducible} if $(W_J, J)$ is irreducible; $J$ is called a \emph{component} of $S$ if it is the vertex set of a connected component of the Coxeter diagram corresponding to $(W, S)$.

A subset $J \subseteq S$ is called \textit{spherical} if $W_J$ is finite. The Coxeter system is called \textit{spherical} if $S$ is spherical, and of \emph{finite type} if in addition it is irreducible. If $(W, S)$ is irreducible, then $(W, S)$ (resp.\ $S$) is said to be of \emph{affine} type if $W$ is infinite and virtually abelian; equivalently, its underlying Coxeter diagram is of type $\tilde{A}_n$, $\tilde{B}_n$, $\tilde{C}_n$, $\tilde{D}_n$, $\tilde{E}_6$, $\tilde{E}_7$, $\tilde{E}_8$, $\tilde{F}_4$ or $\tilde{G_2}$ (see \cite[Ch.\ VI, $\S 4$ Theorem $4$]{Bo68}. The irreducible Coxeter system $(W, S)$ (resp.\ $S$) is said to be of \emph{indefinite} type if it is neither spherical nor affine. For $J \subseteq S$ we let $\JaffCox \subseteq J$ denote the union of all the components of $J$ of affine type. We call $J \subseteq S$ \emph{essential} if $J$ is nonempty and has no spherical component. We further set $J^{\perp} := \{ s\in S \mid m_{sj} = 2 \text{ for all } j \in J \}$.

\begin{lemma}{{\cite[Lemma~$3.3$]{CMR22}}}\label{lemma:prelim_essential}
	Let $J \subseteq S$ be essential. Then $N_W(W_J) = W_J \times W_{J^\perp}$ and, if $w\in W$ is such that $w\inv J w \subseteq S$, then $w\in W_{\Jperp}$.
\end{lemma}

The subgroups of $W$ of the form $W_J$ for some $J\subseteq S$ are called \emph{standard parabolic subgroups}, and their conjugates \emph{parabolic subgroups}. A parabolic subgroup is called \emph{spherical} if it is finite. The intersection $\Pc(w)$ of all parabolic subgroups containing an element $w \in W$ is again a parabolic subgroup, called the \emph{parabolic closure} of $w$. A nontrivial element $w\in W$ is called \emph{straight} if $\ell(w^n)=n\ell(w)$ for all $n\in\NN$. An element $w\in W$ is called \emph{cyclically reduced} if $\ell(w) = \min\{ \ell(vwv\inv) \mid v\in W \}$. We note that straight elements are cyclically reduced by \cite[Lemma~$4.1$]{Ma14b}. For $w\in W$, a decomposition $w = s_1 \cdots s_n$ with $s_1, \ldots, s_n \in S$ and $\ell(w) = n$ is called \emph{reduced}. The \emph{support} $\supp(w)$ of $w\in W$ is the subset $J$ of elements in $S$ appearing in some (equivalently, any) reduced decomposition of $w$.

\begin{lemma}\label{Lemma: straight element J essental}
	Let $w\in W$ be straight and let $J = \supp(w)$. Then $\Pc(w) = W_J$ and $J$ is essential. Moreover, if $J'$ is a component of $J$ and $w=w_1w_2$ with $w_1\in W_{J'}$ and $w_2\in W_{J\setminus J'}$, then $w_1$ is straight and $\supp(w_1) = J'$.
\end{lemma}
\begin{proof}
	The second statement is clear. It also implies that $J$ is essential. Moreover, the parabolic subgroup $\Pc(w)$ is standard by \cite[Proposition~$4.2$(ii)]{CF10} since $w$ is cyclically reduced, and hence $\Pc(w) = W_J$.
\end{proof}

\subsection{Standard elements}\label{Subsection: Standard elements}

Suppose $(W, S)$ is of irreducible indefinite type, and let $w\in W$ with $\Pc(w)=W$. By \cite[Theorem~C$(1)$]{Ma23}, there exists a largest spherical parabolic subgroup $P_w^{\max}$ which is normalized by $w$, and we have $P_w^{\max}=P_{w^n}^{\max}$ for all $n\in\NN^*$. Following \cite{Ma23}, the element $w\in W$ is called \emph{standard} if it is cyclically reduced and $P_w^{\max}$ is standard.

Back to a general Coxeter system $(W, S)$, let $w\in W$ be such that $\Pc(w)=W_J$ for some essential $J\subseteq S$. Let $J_1, \ldots, J_k$ be the components of $J$ of indefinite type and let $J_0=\JaffCox$, so that $J=J_0\cup J_1\cup\cdots \cup J_k$. Let $w_i \in W_{J_i}$ such that $w = w_0 \cdots w_k$. We call $w$ \emph{standard} if $w_i$ is standard for each $i= 1,\ldots,k$. In this case we let $I_i \subseteq J_i$ with $P_{w_i}^{\max} = W_{I_i}$ for all $i= 1,\ldots,k$ and define $\Jsph^w := I_1 \cup \cdots \cup I_k$.

\begin{lemma}\label{Lemma: infinite order and standard elements}
	Let $w\in W$ be straight and let $J = \supp(w)$. Then there exists $v\in W_J$ such that $v\inv w v$ is straight and standard.
\end{lemma}
\begin{proof}
	WLOG, we may assume that $J$ is of irreducible indefinite type. By \cite[Theorem~C$(2)$]{Ma23}, there exists $v\in W_J$ such that $v\inv w v$ is standard. In particular, it is cyclically reduced and, by \cite[Lemma~$4.2$]{Ma14b} it is straight, too.
\end{proof}

\subsection{Buildings}

A \textit{building of type $(W, S)$} is a pair $\mc{B} = (\mc{C}, \delta)$ where $\mc{C}$ is a nonempty set and where $\delta: \mc{C} \times \mc{C} \to W$ is a \textit{distance function} satisfying the following axioms, where $x, y\in \mc{C}$ and $w = \delta(x, y)$:
\begin{enumerate}[label=(Bu\arabic*)]
	\item $w = 1_W$ if and only if $x=y$;
	
	\item if $z\in \mc{C}$ satisfies $s := \delta(y, z) \in S$ then $\delta(x, z) \in \{w, ws\}$, and if, furthermore, $\ell(ws) = \ell(w) +1$ then $\delta(x, z) = ws$;
	
	\item if $s\in S$, there exists $z\in \mc{C}$ such that $\delta(y, z) = s$ and $\delta(x, z) = ws$.
\end{enumerate}
The elements of $\mc{C}$ are called \textit{chambers}. Two chambers $x, y \in \mc{C}$ are called \textit{adjacent} if $\delta(x, y) = s$ for some $s\in S$ (in which case $x$ and $y$ are said to be \textit{$s$-adjacent}). A \textit{gallery} from $x$ to $y$ is a sequence of chambers $(x = x_0, \ldots, x_k = y)$ such that for each $1 \leq l \leq k$, the chambers $x_{l-1}$ and $x_l$ are $s_l$-adjacent for some $s_l\in S$. The number $k$ is then called the \textit{length} of the gallery and the tuple $(s_1,\ldots,s_k)$ its \emph{type}. A gallery from $x$ to $y$ of length $k$ is called \textit{minimal} if there is no gallery from $x$ to $y$ of length $<k$; in this case we write $\dist(x, y) = k$. Assigning to each gallery its type yields a 1--1 correspondence between minimal galleries from $x$ to $y$ and reduced decompositions of $w=\delta(x,y)$  (cf.\ \cite[Exercise~5.20]{AB08}); in particular, $\dist(x, y) = \ell(\delta(x, y))$ for all $x, y \in \mc{C}$.

\subsection{Projections and residues}\label{Subsection: Projections and residues}

Let $\mc{B} = (\mc{C}, \delta)$ be a building of type $(W, S)$. Given a subset $J \subseteq S$ and $x\in \mc{C}$, the \textit{$J$-residue of $x$} is the set $R_J(x) := \{y \in \mc{C} \mid \delta(x, y) \in W_J \}$. Each $J$-residue is a building of type $(W_J, J)$ with the distance function induced by $\delta$ (cf.\ \cite[Corollary~$5.30$]{AB08}). A \textit{residue} is a subset $R$ of $\mc{C}$ such that there exist $J \subseteq S$ and $x\in \mc{C}$ with $R = R_J(x)$. The subset $J$ is uniquely determined by $R$ and is called the \textit{type} of $R$; the cardinality of $J$ is the \textit{rank} of $R$. A residue is called \textit{spherical} if its type is a spherical subset of $S$. A \textit{panel} is a residue of rank $1$. An \textit{$s$-panel} is a panel of type $\{s\}$ for $s\in S$. For $c\in \mc{C}$ and $s\in S$ we put $\mc{P}_s(c) := R_{\{s\}}(c)$. The building $\mc{B}$ is called \textit{thick}, if each panel of $\mc{B}$ contains at least three chambers.

Given $x\in \mc{C}$ and a residue $R \subseteq \mc{C}$, there exists a unique chamber $z\in R$ such that $\dist(x, y) = \dist(x, z) + \dist(z, y)$ for each $y\in R$ (cf.\ \cite[Proposition~5.34]{AB08}). The chamber $z$ is called the \textit{projection of $x$ onto $R$} and is denoted by $\proj_R x$. Moreover, if $z = \proj_R x$ we have $\delta(x, y) = \delta(x, z) \delta(z, y)$ for each $y\in R$.

For two residues $R$ and $T$ of $\mc{B}$ we define $\proj_T R := \{ \proj_T r \mid r \in R \}$. The residues $R$ and $T$ are called \emph{parallel} if $\proj_T R = T$ and $\proj_R T = R$. In this case, the element $\delta(r,\proj_Tr)\in W$ is independent of the choice of $r\in R$ and is denoted $\delta(R,T)$. Let $R_J$ and $R_K$ be two residues of $\mc{B}$ of type $J$ and $K$, respectively. Then $\proj_{R_K} R_J$ is again a residue (contained in $R_K$) and is of type $K \cap w\inv J w$ for some $w\in W$. If, additionally, $R_J$ and $R_K$ are parallel, then $K = w\inv J w$ for some $w\in W$ (cf.\ \cite[Proposition~5.37(1)]{AB08}).

A subset $\Sigma \subseteq \mc{C}$ is called \textit{convex} if for any two chambers $c, d \in \Sigma$ and any minimal gallery $(c_0 = c, \ldots, c_k = d)$ we have $c_i \in \Sigma$ for all $i=0,\ldots,k$. We note that residues are convex (cf.\ \cite[Example~5.44(b)]{AB08}). We denote the convex hull of a set $X \subseteq \mc{C}$ by $\conv(X)$. A subset $\Sigma \subseteq \mc{C}$ is called \textit{thin} if $P \cap \Sigma$ contains exactly two chambers for each panel $P \subseteq \mc{C}$ which meets $\Sigma$. An \textit{apartment} is a nonempty subset $\Sigma \subseteq \mc{C}$, which is convex and thin.

\subsection{Automorphisms}

Let $\mc{B} = (\mc{C}, \delta)$ and $\mc{B}' = (\mc{C}', \delta')$ be two buildings of type $(W, S)$. An \emph{isometry} between $\mc{X} \subseteq \mc{C}$ and $\mc{X}' \subseteq \mc{C}'$ is a bijection $\phi: \mc{X} \to \mc{X}'$ such that $\delta'(\phi(x), \phi(y)) = \delta(x, y)$ for all $x, y \in \mc{X}$. In this case $\mc{X}$ and $\mc{X}'$ are called \emph{isometric}. An isometry $\phi: \mc{C} \to \mc{C}$ is called a \emph{type-preserving automorphism} of $\mc{B}$. We denote their set by $\Aut_0(\mc{B})$.

Let $G$ be a group acting on $\mc{B} = (\mc{C}, \delta)$ by type-preserving automorphisms. We say that this action is \emph{Weyl-transitive} if for each $w\in W$, the group $G$ acts transitively on the set of ordered pairs $(c, d) \in \mc{C} \times \mc{C}$ with $\delta(c, d) = w$.

\begin{example}\label{Example: Coxeter building}
	We define $\delta: W \times W \to W, (x, y) \mapsto x^{-1}y$. Then $\Sigma(W, S) := (W, \delta)$ is a building of type $(W, S)$, which we call the \emph{Coxeter building} of type $(W,S)$ (or \emph{Coxeter complex} of $(W, S)$). The group $W$ acts faithfully on $\Sigma(W, S)$ by multiplication from the left, i.e.\ $W \leq \Aut_0(\Sigma(W, S))$. Moreover, the stabilizers of residues under this action are precisely the parabolic subgroups of $W$. The chamber distance $\dist$ coincides with the usual distance in the Cayley graph of $(W, S)$.
\end{example}

\begin{lemma}\label{Lemma: projection and isometry commute}
	Let $\mc{B} = (\mc{C}, \delta)$ be a building of type $(W, S)$ and let $\phi \in \Aut_0(\mc{B})$. Then $\phi(\proj_R x) = \proj_{\phi(R)} \phi(x)$ for all $x\in \mc{C}$ and each residue $R \subseteq \mc{C}$. In particular, if $\phi$ fixes $x, y \in \mc{C}$, then it fixes each chamber on a minimal gallery from $x$ to $y$.
\end{lemma}
\begin{proof}
	The first part is a consequence of the uniqueness of projections onto residues. The second part follows from the first by induction on $\dist(x, y)$.
\end{proof}

\begin{lemma}\label{Lemma: Fixator of orthogonal residues}
	Let $\mc{B} = (\mc{C}, \delta)$ be a building of type $(W, S)$ and let $c\in \mc{C}$. For $J \subseteq S$ and $J' \subseteq \Jperp$, we have $\Fix_{\Aut_0(\mc{B})}(R_J(c)) \cap \Fix_{\Aut_0(\mc{B})}(R_{J'}(c)) \subseteq \Fix_{\Aut_0(\mc{B})}(R_{J \cup J'}(c))$.
\end{lemma}
\begin{proof}
	Let $d\in R_{J \cup J'}(c)$. Set $d_J:=\proj_{R_J(c)}d$ and $d_{J'}:=\proj_{R_{J'}(c)}d$. Since $\delta(d_J,d)\in W_{J'}$ and $\delta(d,d_{J'})\in W_J$, the concatenation of a minimal gallery from $d_J$ to $d$ with a minimal gallery from $d$ to $d_{J'}$ has reduced type (as $W_{J\cup J'}=W_J\times W_{J'}$) and is thus minimal. The claim then follows from Lemma~\ref{Lemma: projection and isometry commute}.
\end{proof}

\subsection{Roots}

A \textit{reflection} is an element of $W$ that is conjugate to an element of $S$. For $s\in S$ we let $\alpha_s := \{ w\in W \mid \ell(sw) > \ell(w) \}$ be the \textit{simple root} corresponding to $s$. A \textit{root} is a subset $\alpha \subseteq W$ such that $\alpha = v\alpha_s$ for some $v\in W$ and $s\in S$. We denote the set of all roots by $\Phi := \Phi(W, S)$. The set $\Phi_+ = \{ \alpha \in \Phi \mid 1_W \in \alpha \}$ is the set of all \textit{positive roots} and $\Phi_- = \{ \alpha \in \Phi \mid 1_W \notin \alpha \}$ is the set of all \textit{negative roots}. For each root $\alpha \in \Phi$ we set $-\alpha := W \setminus \alpha\in\Phi$ and we denote the unique reflection interchanging $\alpha$ and $-\alpha$ by $r_{\alpha} \in W \leq \Aut_0(\Sigma(W, S))$: if $\alpha = v\alpha_s$ for some $v\in W$ and $s\in S$, then $r_{\alpha}=vsv\inv$. For a root $\alpha \in \Phi$ we define the \emph{support of $\alpha$} as $\supp(\alpha) := \supp(r_{\alpha})$. For $\alpha \in \Phi$ we denote by $\partial \alpha$ the set of all panels of $\Sigma(W, S)$ stabilized by $r_{\alpha}$ and call it the \emph{wall} associated with $\alpha$.

\begin{lemma}\label{Lemma: support of a root irreducible}
	Let $\alpha \in \Phi$ be a root and let $J \subseteq S$. Then the following hold:
	\begin{enumerate}
		\item $\supp(\alpha) \subseteq J$ if and only if $\alpha = w \alpha_s$ for some $w \in W_J$ and $s\in J$.
		
		\item $\supp(\alpha)$ is irreducible.
	\end{enumerate}
\end{lemma}
\begin{proof}
	(1) Write $\alpha = w\alpha_s$ for some $w\in W_J$ and $s\in J$, so that  $r_{\alpha} = w s w\inv \in W_J$. Then $\supp(\alpha) = \supp(r_{\alpha}) \subseteq J$. Conversely, assume that $\supp(\alpha) \subseteq J$. 
	
	By definition, $r_{\alpha} \in W_J$. Let $\beta \in \{\alpha, -\alpha\}$ with $W_J \cap \beta \neq \emptyset$. Note that $W_J \not\subseteq \beta$, for otherwise we would have $r_{\beta}=r_{\alpha}\in \beta$. But $r_{\beta} = r_{\beta} \cdot 1_W \in r_{\beta} \beta = -\beta$ and hence $r_{\beta} \in \beta \cap (-\beta) = \emptyset$, a contradiction. Thus there is a panel $P \in \partial \beta$ with $P\subseteq W_J$. In particular, there exist $w \in W_J$ and $s\in J$ with $\alpha = w\alpha_s$.
	
	(2) Let $J = \supp(\alpha)$. By $(1)$ there exist $w\in W_J$ and $s\in J$ with $\alpha = w\alpha_s$. Let $I \subseteq J$ be the component of $J$ with $s\in I$. Let $w' \in W_I$ and $w'' \in W_{J \setminus I}$ with $w = w' w''$. Then $r_{\alpha} = wsw^{-1} = w' s (w')^{-1} \in W_I$. This implies that $J = \supp(\alpha) \subseteq I$ and hence that $J$ is irreducible.
\end{proof}

For a subset $Z \subseteq W$ we denote by $\conv(Z)$ its convex hull (with respect to the chamber distance $\dist$); equivalently, $\conv(Z)$ is the intersection of all roots containing $Z$ (see \cite[Proposition~3.94]{AB08}).

\subsection{RGD systems}\label{subsection:RGDsystems}

An \textit{RGD system of type $(W, S)$} is a pair $\mathcal{D} = \left( G, \left( U_{\alpha} \right)_{\alpha \in \Phi}\right)$ consisting of a group $G$ together with a family of subgroups $U_{\alpha}$ (called \textit{root groups}) indexed by the set of roots $\Phi = \Phi(W, S)$, satisfying some axioms (for the precise axioms we refer the reader to \cite[Sections~7.8 and 8.6]{AB08}). We will only need certain properties of RGD systems which we now discuss.

Let $\mathcal{D} = (G, (U_{\alpha})_{\alpha \in \Phi})$ be an RGD system of type $(W, S)$. We set $T := \bigcap_{\alpha \in \Phi} N_G(U_{\alpha})$, $U_{\pm} := \langle U_{\alpha} \mid \alpha \in \Phi_{\pm} \rangle$ and $B := T U_+ = T \ltimes U_+$. There exists a subgroup $N$ of $G$ normalizing $T$ such that $N/T \cong W$ and such that $(B,N)$ is a \emph{BN-pair} for $G$ (see \cite[Section~6.2.6]{AB08}) --- when no confusion is possible, we will often make the slight abuse of notation of viewing $W$ as a subgroup of $G$. In particular, there exists a thick building $\mc{B}=\mc{B}(\mathcal{D})$ with chamber set $G/B$ on which $G$ acts Weyl-transitively, by left multiplication (see \cite[Chapter~$8$]{AB08} for more information). From now on we will relax notation and also denote by $\mc{B}$ the set of chambers of $\mc{B}$. The trivial coset $c:=1_GB$ in $G/B$ is called the \emph{fundamental chamber} (thus $B = \Fix_G(c)$). The \emph{fundamental apartment} $\Sigma:=\{ wc \ | w\in W\}\subseteq \mc{B}$ of $\mc{B}$ can be $W$-equivariantly identified with $\Sigma(W,S)$ (with $c\in\Sigma$ corresponding to $1_W\in\Sigma(W,S)$), and we have $N = \Stab_G(\Sigma)$ and $T = \Fix_G(\Sigma)$. We will then also refer to the roots of $(W,S)$ as \emph{roots of $\Sigma$}, and keep the notation $\Phi$ to denote their set. The action of the root groups on $\mc{B}$ has the following properties (cf.\ \cite[Section~$8.6.4$]{AB08}), where $\alpha \in \Phi$:
\begin{itemize}
	\item $U_{\alpha}$ fixes the root $\alpha$ of $\Sigma$ as well as each \emph{interior panel} $P$ of $\alpha$, i.e.\ each panel $P$ with the property $\vert P \cap \alpha \vert = 2$.
	
	\item For each \emph{boundary panel} $P$ of $\alpha$ (i.e.\ a panel $P$ with the property $\vert P \cap \alpha \vert = 1$) the action of the root group on $P \setminus \{p\}$ is simply transitive, where $p$ is the unique chamber contained in $P \cap \alpha$.
\end{itemize}

\section{Completions of RGD systems}\label{section:completionsRGD}

Throughout this section, we fix an RGD system $\mc{D}=(G,(U_{\alpha})_{\alpha\in\Phi})$ of type $(W,S)$, with associated building $\mc{B}=\mc{B}(\mc{D})$, fundamental chamber $c$ and fundamental apartment $\Sigma$. We always equip the automorphism group $\Aut_0(\mc{B})$ with the permutation topology: a basis of identity neighborhoods is given by pointwise stabilizers of finite sets of chambers.
We moreover assume that the root groups $U_{\alpha}$ are finite, to ensure that ``completing'' $G$ yields a tdlc group, for which the notion of nub of an automorphism is defined. Depending on the context, there are actually several natural ``completions'' of $G$ that can be considered (see Examples~\ref{Example: Building completion} and \ref{Example: Scheme-theoretic completion} below), and we start by isolating their common characteristics so as to be able to treat them in a uniform way.

\begin{definition}
	We call a topological group $H$ a (\emph{positive}) \emph{completion} of $\mc{D}$ if $H$ is a first-countable tdlc group with a continuous action $H \to \Aut_0(\mc{B})$, and if there exists an injective group morphism $\iota\co G\to H$ such that the diagram
	\[
	\begin{tikzcd}
		G \arrow{r}{\iota} \arrow[swap]{dr}{} & H \arrow{d}{} \\
		& \Aut_0(\mc{B})
	\end{tikzcd}
	\]
	is commutative (we then identify $G$ with a subgroup of $H$). We set $H_d:=\Fix_H(d)$ for any chamber $d$ of $\mc{B}$.
\end{definition}

\begin{definition}\label{definition:admissible_subgroup}
	Let $H$ be a completion of $\mc{D}$. We call a subgroup $V$ of $H$ a \emph{unipotent subgroup of $H$} if it satisfies the following two conditions:
	\begin{enumerate}
		\item[(US1)] $U_+\leq V\leq H_c$ and $V$ is normalized by $T$;
		\item[(US2)] $V=U_{\alpha_s}\ltimes V_{(s)}$ for all $s\in S$, where $V_{(s)}:=V\cap sVs\inv$.
	\end{enumerate}
	If, in addition, $V$ is compact and open, then we call $V$ an \emph{admissible subgroup of $H$}.
\end{definition}

The terminology ``unipotent'' is motivated by the following observation.

\begin{lemma}\label{lemma:basic_prop_unipotent}
Let $V$ be a unipotent subgroup of $H$ and $s\in S$. Then
$$V_{(s)}=\Fix_V(\{c,sc\})=\Fix_V(\mc{P}_s(c)).$$ 
In particular, if $v\in V$ fixes two chambers of $\mc{P}_s(c)$, then it fixes $\mc{P}_s(c)$ pointwise.
\end{lemma}
\begin{proof}
Let $g\in V$, and write $g=u_{\alpha_s}v$ with $u_{\alpha_s}\in U_{\alpha_s}$ and $v\in V_{(s)}$. Let $d\in \mc{P}_s(c)\setminus\{c\}$, and write $d=u_ssc$ with $u_s\in U_{\alpha_s}$ (recall that $U_{\alpha_s}$ acts simply transitively on $\mc{P}_s(c)\setminus\{c\}$). Then $gd=u_{\alpha_s}d$ as $u_s$ and $s$ normalize $V_{(s)}\subseteq H_c$ by assumption. Hence $gd=d$ if and only if $u_{\alpha_s}=1$ if and only if $g\in V_{(s)}$. We conclude that $V_{(s)} = \Fix_V(\{c, d\})$ for all $d\in \mc{P}_s(c) \setminus \{c\}$. The lemma follows.
\end{proof}

Here is an additional property of unipotent subgroups of $H$.
\begin{lemma}\label{Lemma: U_+ cap G_n.c}
Let $V$ be a unipotent subgroup of $H$ and $w\in W$. Then $\Fix_V(wc) \leq w V w\inv$.
\end{lemma}
\begin{proof}
We prove the claim by induction on $\ell := \ell(w)$. If $\ell = 0$, there is nothing to show. Assume now that $\ell >0$ and let $u \in \Fix_V(wc)$. Let $s\in S$ with $\ell(sw) < \ell(w)$. Then $sc$ is on a minimal gallery from $c$ to $wc$. Since $u$ fixes $c$ and $wc$ by assumption, it fixes $sc$ by Lemma~\ref{Lemma: projection and isometry commute}. Hence $u\in V_{(s)}$ by Lemma~\ref{lemma:basic_prop_unipotent}. In particular, $sus\inv\in V_{(s)}\subseteq V$. We then deduce from the induction hypothesis that 
\[ sus\inv \in V \cap s\Fix_V(wc)s\inv = \Fix_V(swc) \leq sw V (sw)\inv,\]
so that $u\in wVw\inv$, as desired.
\end{proof}

Here is the main source of examples of unipotent subgroups of $H$ that we will want to consider (see Examples~\ref{Example: Building completion} and \ref{Example: Scheme-theoretic completion} below).
\begin{lemma}\label{lemma:criterion_unipotent}
	Let $V \leq H$ be a subgroup satisfying (US1). Suppose that for each $s\in S$ there exists $N_s \trianglelefteq V$ such that $sN_s s\inv = N_s$ and $V = U_{\alpha_s} N_s$. Then $V$ is a unipotent subgroup of $H$.
		
	In particular, the group $U_+$ and its closure $\overline{U_+}$ in $H$ are unipotent subgroups of $H$.
\end{lemma}
\begin{proof}
	Let $s\in S$. We have $U_{\alpha_s}\cap V_{(s)}\subseteq U_{\alpha_s}\cap\Fix_V(sc)=\{1\}$. To see that $U_{\alpha_s}$ normalizes $V_{(s)}$, note first that $N_s \subseteq V_{(s)}$. Let now $v\in V_{(s)}$ and $u_s\in U_{\alpha_s}$. Then $u_svu_s\inv=u_{\alpha_s}n_s$ for some $u_{\alpha_s}\in U_{\alpha_s}$ and $n_s\in N_s$, and hence $u_s\inv u_{\alpha_s}u_s=v\cdot u_s\inv n_s\inv u_s\in U_{\alpha_s}\cap V_{(s)}=\{1\}$. Thus $u_{\alpha_s}=1$ and $u_svu_s\inv = n_s \in N_s\subseteq V_{(s)}$, as desired. Since $V = U_{\alpha_s} N_s\subseteq U_{\alpha_s}V_{(s)}$, we conclude that $V=U_{\alpha_s}\ltimes V_{(s)}$, and hence $V$ is a unipotent subgroup of $H$.
	
	For $s\in S$ we define $N(s) := \langle u_{\alpha_s} u_{\alpha} u_{\alpha_s}^{-1} \mid \alpha \in \Phi_+ \setminus \{ \alpha_s \}, u_{\alpha} \in U_{\alpha}, u_{\alpha_s} \in U_{\alpha_s} \rangle$. If $V = U_+$, then $N_s := N(s) \trianglelefteq U_+$ is normalized by $s$ (see e.g.\ \cite[Lemme~6.2.1(iv)]{Re03}) and satisfies $V = U_{\alpha_s} N_s$. If $V = \overline{U_+}$, then $N_s := \overline{N(s)} \trianglelefteq \overline{U_+}$ is also normalized by $s$ and satisfies $V = U_{\alpha_s} N_s$. Hence $U_+$ and $\overline{U_+}$ are unipotent subgroups of $H$.
\end{proof}

\begin{example}[Geometric completion]\label{Example: Building completion}
The most natural completion of a general RGD system $\mc{D}$ is its geometric completion, that is, its completion with respect to the building topology (see \cite{CR09}). More precisely, denote for each $n\in\NN$ by $K_n$ the pointwise stabilizer in $U_+$ of the ball centered at $c$ of radius $n$ in $\mc{B}$. The (positive) \emph{geometric completion} $G^{\geo}$ of $G$ is the Hausdorff completion of $G$ with respect to the filtration $(K_n)_{n\in\NN}$ (see e.g.\ \cite[Exercise~8.5]{Ma18}). In other words, $G^{\geo}$ is the group of equivalence classes of Cauchy sequences in $G$, where $(g_m)_{m\in\NN}\subseteq G$ is a \emph{Cauchy sequence} if for all $M\in\NN$ we have $g_m\inv g_n\in K_M$ for all large enough $m,n$.\\
Since $\bigcap_{n\in\NN}K_n=\{1\}$, the natural map $\iota\co G\to H:=G^{\geo}$ is injective. If the root groups of $\mc{D}$ are finite, the group $H$ is a first countable tdlc group with basis of identity neighbourhoods its compact open subgroups $\overline{K_n}$, $n\in\NN$ (see \cite[Exercises~8.5, 8.6]{Ma18}). By Lemma~\ref{lemma:criterion_unipotent}, the subgroup $V:=\overline{U_+}=\overline{K_0}$ of $H$ is an admissible subgroup of $H$. If, in addition, the torus $T$ is finite, then $H_c$ is compact open.
\end{example}

\begin{example}[Maximal Kac--Moody groups]\label{Example: Scheme-theoretic completion}
When $(G,(U_{\alpha})_{\alpha\in\Phi})$ is the RGD system of a (minimal) Kac--Moody group $G=\G_A(k)$ over a finite field $k$, there are several natural ``completions'' of $G$, called maximal Kac--Moody groups (see \cite[Chapter~8]{Ma18}). For the purpose of computing the nub of an element, however, it is sufficient to focus on the \emph{scheme-theoretic completion} $\G^{\pma}_A(k)$ of $G$ (see \S\ref{Section: Maximal Kac--Moody groups}), as well as on the closure $\overline{G}$ of $G$ in $\G^{\pma}_A(k)$ (if the characteristic of $k$ is small, $G$ is not necessarily dense in $\G^{\pma}_A(k)$). Indeed, all other completions are obtained from these ones as quotients under an open continuous morphism with compact kernel (see \cite[Theorem~8.95]{Ma18}), and one can then use Remark~\ref{Remark: nub in quotients} to compute the nub.\\
The group $\G^{\pma}_A(k)$ is an amalgamated product of $G$ with a group $\U^{\map}_A(k)$ (see \S\ref{Section: Maximal Kac--Moody groups}) which is an admissible subgroup of $\G^{\pma}_A(k)$. Similarly, the closure $\overline{U_+}\subseteq \U^{\map}_A(k)$ of $U_+$ in $\G^{\pma}_A(k)$ is an admissible subgroup of $\overline{G}$ (see Remark~\ref{remark: schematic-completion is Kac--Moody completion}). Moreover, chamber stabilizers in both $\G^{\pma}_A(k)$ and $\overline{G}$ are compact open.
\end{example}

We conclude this section with a general observation ensuring that completions of RGD systems with compact chamber stabilizers are transitive on the set of apartments of $\mc{B}$ (for readers who are more familiar with the simplicial approach to buildings, the set of apartments as defined in \S\ref{Subsection: Projections and residues} is the \emph{complete apartment system} of $\mc{B}$). 

\begin{proposition}\label{Proposition: Stabilizer by U_+ fixes apartment}
Let $H$ be a completion of $\mc{D}$ and assume that $H_c$ is compact. Let $\Sigma'$ be an apartment containing $c$. Then there exists $u \in \overline{\langle U_{\alpha} \mid \alpha\supseteq \Sigma\cap\Sigma' \rangle}$ with $u\Sigma'=\Sigma$.
\end{proposition}
\begin{proof}
Set for short $\sigma:=\Sigma\cap\Sigma'$ and $U_{\sigma}:=\langle U_{\alpha} \mid \alpha\supseteq \Sigma\cap\Sigma' \rangle$. Reasoning inductively, it is sufficient to show that if $d$ is a chamber of $\Sigma'$ adjacent to a chamber of $\sigma$, then there exists $u\in U_{\sigma}$ such that $u\Sigma'\cap\Sigma\supseteq \sigma\cup\{d\}$. Indeed, this will provide a sequence $(u_n)_{n\in\NN}\subseteq U_{\sigma}$ such that $u_n\Sigma'\cap\Sigma\supseteq B(c,n)\cap\Sigma$ for all $n\in\NN$, where $B(c, r) = \{ x\in \mc{B} \mid \dist(c, x) \leq r \}$. Since $\overline{U_{\sigma}} \leq H_c$ is compact, $(u_n)_{n \in \NN}$ then subconverges to some $u \in \overline{U_{\sigma}}$. As the action $H \to \Aut_0(\mc{B})$ is continuous, $u$ is as desired.

WLOG, we may assume that $d\notin \sigma$. Since $\sigma$ is convex, and hence an intersection of roots, there is a root $\alpha\supseteq\sigma$ such that $d\in -\alpha$, and we choose $u\in U_{\alpha}$ such that $ud\in\Sigma$, as desired.
\end{proof}

\section{Computing the nub from straight and standard elements of $W$}\label{Section: Reduction steps}
	
Let $\mc{D} = (G, (U_{\alpha})_{\alpha \in \Phi})$ be an RGD system of type $(W, S)$ with finite root groups. Let $H$ be a completion of $\mc{D}$ acting on $\mc{B} := \mc{B}(\mc{D})$ with compact open chamber stabilizers. As before, we let $c$ be the fundamental chamber and $\Sigma$ be the fundamental apartment of $\mc{B}$. Fix a section $W \to N:w\mapsto\widetilde{w}$ of the quotient map $N \to W$ (see \S\ref{subsection:RGDsystems}). In this section we show that it is enough to compute the nub of elements $\widetilde{w}\in N$ where $w\in W$ is straight and standard.

\begin{remark}\label{Remark: topologically periodic elements have trivial nub}
	Suppose $g\in H$ stabilizes a spherical residue $R$ of $\mc{B}$. Then, as $\Stab_H(R)$ is compact, $g$ is \emph{topologically periodic}, i.e.\ $\overline{\la g \ra}$ is compact. It is a well-known fact that topologically periodic elements have trivial contraction group (cf.\ \cite[Lemma~$3.5$]{BRR08}), hence $\nub_H(g) = \{1\}$. Thus it suffices to consider elements which do not stabilize any spherical residue.
\end{remark}

The next lemma is a generalization of \cite[Lemma~$3.5$]{BW04}.

\begin{lemma}\label{Lemma: Generalization of Lemma 3.5 in BW04}
	Let $\Gamma$ be a topological group (not necessarily Hausdorff). Let $d, v \in \Gamma$ and suppose that $\langle d^n v d^{-n} \mid n\in \NN \rangle$ is relatively compact. Then $\con(dv) = \con(d)$.
\end{lemma}
\begin{proof}
	We set for short $C := \langle d^n v d^{-n} \mid n\in \NN \rangle$. We first show by induction on $m \in \NN$ that there exists $v_m \in C$ with $(dv)^m = v_m d^m$. For $m=0$ choose $v_0 = 1_G \in C$. Suppose now $m>0$. We have $(dv)^m = dv v_{m-1} d^{m-1} = (dv v_{m-1} d^{-1}) d^m$ by induction hypothesis. Since $vv_{m-1} \in C$, the claim follows with $v_m := d v v_{m-1} d\inv\in C$.
	
	Now we prove the lemma. Let $g \in \con(d)$ and let $\mathcal{O}$ be an open neighborhood of $1_G$. By \cite[Theorem~$4.9$]{HR79} there exists an open neighborhood $\mc{O}'$ of $1_G$ such that $c \mc{O}' c\inv \subseteq \mc{O}$ for all $c\in C$. For all large enough $n \in \NN$, we then have
	\[ (dv)^n g (dv)^{-n} = v_n d^n g d^{-n} v_n^{-1} \in v_n \mathcal{O}' v_n^{-1} \subseteq \mathcal{O}. \]
	Hence $g\in \con(dv)$. For the reverse inclusion let $g\in \con(dv)$ and let $\mc{O}$ be an open neighborhood of $1_G$. As before, the following holds for all $n\in \NN$ large enough:
	\[ d^n g d^{-n} = v_n\inv v_n d^n g d^{-n} v_n\inv v_n = v_n\inv (dv)^n g (dv)^{-n} v_n \in v_n\inv \mc{O}' v_n \subseteq \mc{O}. \]
	Hence $g\in \con(d)$. This concludes the proof.
\end{proof}

Recall from \S\ref{Subsection: Projections and residues} the definition of parallel residues.

\begin{proposition}\label{prop: The nub only for straight and standard elements}
	Let $g\in H$. Suppose that $g$ does not stabilize any spherical residue. Then there exists $w\in W$ straight and standard and $h\in H$ such that
	\[ \con(g) = h \con(\widetilde{w}) h\inv \quad \text{and} \quad \con(g\inv) = h \con(\widetilde{w}\inv) h\inv. \]
	In particular, $\nub(g) = h \nub(\widetilde{w}) h\inv$.
\end{proposition}
\begin{proof}
	As $g$ does not stabilize any spherical residue, there exists a spherical residue $R$ such that $R$ and $g^i R$ are parallel for all $i\in\ZZ$ and such that $x := \delta(r, \proj_{gR} r)$ ($r\in R$) is straight (cf.\ \cite[Theorem~$4.7$ and Theorem~$5.5$]{BPR19}). Let $r\in R$.

We claim that there exists $i \in \{1, \ldots, \vert R \vert\}$ with $\delta(r, g^i r) = x^i$. Indeed, set $r_0:=r$ and for each $i\geq 1$, set $r_i:=\proj_{g^iR}r_{i-1}\in g^iR$. Thus $\delta(r_{i-1},r_i)=\delta(g^{i-1}R,g^iR)=\delta(R,gR)=x$ for all $i\geq 1$. Since $x$ is straight, the set $\{r_i \ | \ i\in\NN\}$ is then contained in an apartment by \cite[Theorem~$5.73$]{AB08}. It then follows from \cite[Proposition~21.55]{MPW15} that $\proj_{g^iR}r_j=r_i$ for all $i,j$. Set $n := \vert R \vert$. Since $\vert \{ \proj_{g^n R} g^i r \mid 0 \leq i \leq n \} \vert \leq \vert R \vert$, there exist $0 \leq j < k \leq n$ such that $\proj_{g^n R} g^j r = \proj_{g^n R} g^k r$. Hence $r_{n-j}=g^{-j}\proj_{g^n R} g^j r=\proj_{g^{n-j} R} g^{k-j} r$. Set $i:=k-j$. As the residues $g^{n-j} R$ and $g^{i} R$ are parallel, we conclude that $g^{i} r=\proj_{g^{i} R} r_{n-j}=r_i$, so that $\delta(r, g^{i} r) = \delta(r, r_i) = x^{i}$.

	Since $x^i$ is straight, $L := \{ g^{iz} r \mid z\in \ZZ \}$ is contained in an apartment $\Sigma'$ by \cite[Theorem~$5.73$]{AB08}. Since $H$ is transitive on chambers, Proposition~\ref{Proposition: Stabilizer by U_+ fixes apartment} yields some $u\in H$ with $u\Sigma = \Sigma'$ and $uc = r$. Note that for $d\in \Sigma$, the chamber $ud$ is the unique chamber in $\Sigma’$ with the property that $\delta(r, ud) = \delta(c, d)$ (see \cite[Corollary~5.67]{AB08}); in particular, for any $z\in\ZZ$, we have $ux^{iz}c=g^{iz}r$. Set $n:=\widetilde{x}^i\in N$. Then $$un^{\pm 1}u\inv \cdot g^{iz} r=un^{\pm 1} \cdot x^{iz}c=u \cdot x^{i(z\pm 1)}c=g^{i(z\pm 1)}r,$$ and hence $v:=g^{-i} (unu\inv)$ and $v':=g^i(un\inv u\inv)$ belong to $\Fix_H(L)$. Note that $g^{iz} \Fix_H(L) g^{-iz} = \Fix_H(L)$ for each $z\in \ZZ$, which is compact. By Lemma~\ref{Lemma: Generalization of Lemma 3.5 in BW04}, we deduce that $\con(g^i) = \con(g^i v) = \con(u n u\inv) = u \con(n) u\inv$ and $\con(g^{-i}) = \con(g^{-i} v') = \con(u n\inv u\inv) =u \con(n\inv)u\inv$. By Lemma~\ref{Lemma: infinite order and standard elements} there exists $v\in W$ such that $w:=v\inv x v$ is straight and standard. For $h := u \widetilde{v} \in H$ we conclude that
	\[ \con(g^{\pm 1}) = \con( g^{\pm i} ) = u \con(\widetilde{x}^{\pm i}) u\inv = h \con(\widetilde{w}^{\pm i}) h\inv = h \con(\widetilde{w}^{\pm 1}) h\inv. \qedhere \]
\end{proof}

\section{The nub in completions of RGD systems}\label{Section: The nub in completions}

Throughout this section, we fix an RGD system $\left( G, \left( U_{\alpha} \right)_{\alpha \in \Phi} \right)$ of type $(W, S)$ with finite root groups and we set $T = \bigcap_{\alpha \in \Phi} N_G(U_{\alpha})$. As before, we write $\mc{B} := \mc{B}(\mc{D})$ for the associated building of type $(W, S)$. Let $c\in \mc{B}$ be the fundamental chamber and let $\Sigma$ be the fundamental apartment of $\mc{B}$. We also fix a section $W \to N:w\mapsto\widetilde{w}$ of the quotient map $N \to W$. Finally, we let $H$ denote a completion of $\mc{D}$ acting on $\mc{B}$ with compact open chamber stabilizers, and we let $V$ be an admissible subgroup of $H$ (see Definition~\ref{definition:admissible_subgroup}).

\subsection{Tidyness of $V$}

We want to show that for $w\in W$ straight the subgroup $V$ is tidy for $\widetilde{w}$. To do so we first show that $H_c$ is tidy and deduce that $V$ is minimizing.

As $G$ acts Weyl-transitively on $\mc{B}$, for each $s\in S$, all $s$-panels of $\mc{B}$ have the same cardinality. Moreover, as all root groups are finite and act simply transitively on panels, all panels are finite. If $P$ is an $s$-panel, we set $q_s := \vert P \vert -1=|U_{\alpha_s}|$.

The following lemma is analoguous to \cite[Theorem~4.7]{BPR19} and its proof follows essentially the same lines. However, in \cite{BPR19} the completion $H$ is assumed to be a subgroup of $\Aut_0(\mc{B})$. Moreover, they show that the stabilizer of a suitable spherical residue is tidy, while we show that the stabilizer of a chamber is tidy.

\begin{lemma}\label{lemma: Stabilizer is tidy}
	Let $w\in W$ be straight. Then $H_c$ is tidy for $\widetilde{w}$. Moreover, if $w=s_1\cdots s_k$ is a reduced decomposition of $w$, then $s(\widetilde{w}) = q_{s_1} \cdots q_{s_k}$.
\end{lemma}
\begin{proof}
	Lemma~\ref{Lemma: auxiliary results for tidy subgroups}$(1,2)$ implies that $H_c$ is tidy for $\widetilde{w}$ (equivalently, for $\widetilde{w}\inv$) if and only if
	\[ [H_c: H_c \cap  H_{\widetilde{w}^nc}] = [H_c: H_c \cap \widetilde{w}^n H_c \widetilde{w}^{-n}] \overset{!}{=} [H_c: H_c \cap \widetilde{w} H_c \widetilde{w}^{-1}]^n = [H_c: H_c \cap H_{\widetilde{w}c}]^n \]
	holds for all $n\in \NN$.  Let $(c_0 := c, \ldots, c_k := \widetilde{w}c)$ be a minimal gallery. As $w$ is straight, the concatenation of the minimal galleries $(\widetilde{w}^ic_0, \ldots, \widetilde{w}^ic_k)$, $0 \leq i \leq n-1$, is a minimal gallery from $c$ to $\widetilde{w}^nc$, which we denote $(d_0, \ldots, d_{kn})$. By Lemma~\ref{Lemma: projection and isometry commute},
	\begin{align*}
		\left[ H_c: H_c \cap H_{\widetilde{w}^nc} \right] = \left[ H_{d_0}: \bigcap\limits_{i=0}^{kn} H_{d_i} \right] = \prod_{i=1}^{kn} \left[ \bigcap\limits_{j=0}^{i-1} H_{d_j} : \bigcap\limits_{j=0}^i H_{d_j} \right].
	\end{align*}
	Let $P_i$ be the panel containing $d_{i-1}$ and $d_i$ and let $s_i \in S$ be the type of $P_i$. Then $d_{i-1} = \proj_{P_i} d_0$. 
	
	We claim that $\left[ \bigcap_{j=0}^{i-1} H_{d_j}: \bigcap_{j=0}^i H_{d_j} \right] =q_{s_i}$. Indeed, let $p \in P_i\setminus\{d_{i-1}\}$. As $\delta(d_0, d_i) = \delta(d_0, d_{i-1}) s_i = \delta(d_0, p)$ and as $H$ acts Weyl-transitively on $\mc{B}$, there exists $h=h_p\in H_{d_0}$ with $hd_i = p$. Lemma~\ref{Lemma: projection and isometry commute} implies that $hd_{i-1} = h \left( \proj_{P_i} d_0 \right) = \proj_{hP_i} (hd_0) = \proj_{P_i} d_0 = d_{i-1}$. Again by Lemma~\ref{Lemma: projection and isometry commute} we deduce that $h\in \bigcap_{j=0}^{i-1} H_{d_j}$. Note, moreover, that if $p' \in P_i\setminus\{d_{i-1}\}$ is such that $h_{p'}H_{d_i}=h_pH_{d_i}$, then $p=p'$. Thus  $\{h_p \ | \ p \in P_i\setminus\{d_{i-1}\}\}$ is a set of coset representatives for $\bigcap_{j=0}^i H_{d_j}$ in $\bigcap_{j=0}^{i-1} H_{d_j}$, yielding the claim.
	
	We conclude that
	\allowdisplaybreaks
	\begin{align*}
		\prod_{i=1}^{kn} \left[ \bigcap_{j=0}^{i-1} H_{d_j}: \bigcap_{j=0}^i H_{d_j} \right] &= \prod_{i=1}^{kn} q_{s_i} = \left( q_{s_1} \cdots q_{s_k} \right)^n = \left[ H_c: H_c \cap H_{\widetilde{w}c} \right]^n.
	\end{align*}
	Note that the proof also shows that $s(\widetilde{w}\inv) = [H_c : H_c \cap H_{\widetilde{w}c}] =  q_{s_1} \cdots q_{s_k}$, where $(s_1, \ldots, s_k)$ is the type of a minimal gallery from $c$ to $\widetilde{w}c$. We deduce that $s(\widetilde{w}) = q_{s_k} \cdots q_{s_1} = q_{s_1} \cdots q_{s_k}$.
\end{proof}

\begin{proposition}\label{prop: Upper bound for nub}
	Let $w\in W$ be straight. Then $V$ is tidy for $\widetilde{w}$ and $\nub(\widetilde{w}) \leq \bigcap_{z\in \ZZ} \widetilde{w}^z V \widetilde{w}^{-z}$.
\end{proposition}
\begin{proof}
	By Lemma~\ref{lemma: Stabilizer is tidy} we know that $H_c$ is tidy for $\widetilde{w}$ and $s(\widetilde{w}) = q_{s_1} \cdots q_{s_k}$, where $w=s_1\ldots s_k$ is a reduced decomposition of $w$. To show that $V$ is tidy for $\widetilde{w}$ (equivalently, for $\widetilde{w}\inv$) it thus suffices by Theorem~\ref{Theorem: Theorem 3.1 in Wi01} to show that $$[V : V\cap \widetilde{w} V \widetilde{w}\inv]=[\widetilde{w}\inv V \widetilde{w} : \widetilde{w}\inv V \widetilde{w} \cap V] \leq s(\widetilde{w}\inv) = q_{s_1} \cdots q_{s_k}.$$ By Lemma~\ref{Lemma: U_+ cap G_n.c} we have $V \cap H_{\widetilde{w}c} = V \cap \widetilde{w} V \widetilde{w}^{-1}$.
	To conclude, note that the map 
	$$V/(V \cap H_{\widetilde{w}c})\to \{ d\in \mc{B} \mid \delta(c, d) = \delta(c, \widetilde{w}c) \}: h(V \cap H_{\widetilde{w}c})\mapsto h\widetilde{w}c$$
	is injective, and that $\vert \{ d\in \mc{B} \mid \delta(c, d) = \delta(c, \widetilde{w}c) \} \vert = q_{s_1} \cdots q_{s_k}$. The second assertion follows from Lemma~\ref{Lemma: auxiliary results for tidy subgroups}(3).
\end{proof}

\subsection{Upper bound for the nub}\label{subsection:upperbound_generalRGD}

The purpose of this subsection is to give an upper bound for the nub of an element $\widetilde{w}$ with $w\in W$ straight.

\begin{lemma}\label{lemma:centralizerfixres}
Let $w\in W$ be straight and $I\subseteq S$. Suppose that there exists $N\in\NN^*$ such that $w^N$ centralizes $W_I$. Then $\nub(\widetilde{w}) \leq \Fix_H( R_I(c) )$.
\end{lemma}
\begin{proof}
Write $G_I:=\langle U_{\alpha} \ | \ \supp(\alpha)\subseteq I\rangle\leq G$, so that $R_I(c)=G_Ic$. Note that if $s\in I$, then as $w^Nsw^{-N}=s$ by assumption we have $w^N\alpha_s\in\{\pm\alpha_s\}$ and hence $w^{2N}\alpha_s=\alpha_s$. In particular, $\widetilde{w}^{2N}$ normalizes each $U_{\alpha}$ with $\supp(\alpha)\subseteq I$ by Lemma~\ref{Lemma: support of a root irreducible}(1). Since the root groups $U_{\alpha}$ are by assumption finite (and of uniformly bounded size), there thus exists $M\in\NN^*$ such that $\widetilde{w}^{2MN}$ centralizes $G_I$. Hence $gH_cg\inv$ is tidy for $\widetilde{w}^{2MN}$ for each $g\in G_I$ by Lemmas~\ref{Lemma: auxiliary results for tidy subgroups}(c) and \ref{lemma: Stabilizer is tidy}. As $\nub(\widetilde{w})=\nub(\widetilde{w}^{2MN})$ and $\Fix_H( R_I(c) )=\bigcap_{g\in G_I}gH_cg\inv$, the lemma follows.
\end{proof}

\begin{lemma}\label{Lemma: affine translations}
	Let $(W, S)$ be irreducible and affine. Let $w\in W$ be straight. 
	\begin{enumerate}
	\item
	There exist $r\geq 1$ and $t \in W$ with $W = \conv( \{t^z \mid z\in \ZZ\} )$ such that $t$ and $w^r$ commute. 
	\item
	If $\alpha\in\Phi_+$ is such that $w^n\alpha=\alpha$ for some $n\in\NN^*$, then $(W,S)$ is of rank $\geq 3$, and there is some $N\in\NN^*$ depending only on $(W,S)$ such that $w^N\alpha=\alpha$, and $\alpha\supseteq \{w^z \ | \ z\in\ZZ\}$.
	\end{enumerate}
\end{lemma}
\begin{proof}
Since $W$ is of affine type, the geometric realization of $\Sigma(W,S)$ is a Euclidean space, and the (abelian) subgroup $T_0\leq \Aut_0(\Sigma(W,S))$ of translations of $W$ has finite index in $W$ (see \cite[Chapter~10]{AB08}). 

(1) Pick any translation $t\in T_0$ whose axes are not parallel to any wall of $\Sigma(W,S)$. Then $W = \conv( \{t^z \mid z\in \ZZ\} )$ because convex subcomplexes of $\Sigma(W,S)$ are intersections of roots by \cite[Proposition~3.94]{AB08}. Taking $r$ such that $w^r\in T_0$ yields (1).

(2) Let $N:=[W:T_0]$, so that $w^N$ is a translation. For $\alpha\in\Phi_+$, either $\partial\alpha$ contains a translation axis of $w^N$ (in which case $(W,S)$ has rank $\geq 3$ and $w^N\alpha=\alpha$ and $\alpha\supseteq \{w^z \ | \ z\in\ZZ\}$) or $\partial\alpha$ intersects any translation axis of $w^N$ in a single point (in which case  $w^n\alpha\neq\alpha$ for all $n\in\NN^*$). This proves (2).
\end{proof}

\begin{proposition}\label{Proposition: nub contained in fixators}
	Let $w\in W$ be straight and let $J = \supp(w)$. Then:
	\begin{enumerate}
		\item $\nub(\widetilde{w}) \leq \Fix_H( R_{J^{\perp}}(c) )$.
		
		\item $\nub(\widetilde{w}) \leq \Fix_H( R_I(c) )$ for all spherical $I \subseteq S$ with $w \in N_W(W_I)$.
		
		\item $\nub(\widetilde{w}) \leq \Fix_H(\Sigma \cap R_{\JaffCox}(c))$.
	\end{enumerate}
\end{proposition}
\begin{proof}
	(1) and (2) readily follow from Lemma~\ref{lemma:centralizerfixres}. To prove (3), it suffices by Lemma~\ref{Lemma: Fixator of orthogonal residues} to show that $\nub(\widetilde{w}) \leq \Fix_H(\Sigma \cap R_{J'}(c))$ for each component $J'$ of $\JaffCox$. Let thus $J'$ be an affine component of $J$. By Lemma~\ref{Lemma: straight element J essental} there exist $w' \in W_{J'}$ straight and $w'' \in W_{J \setminus J'}$ with $w = w' w''$ and $\supp(w') = J'$. By Lemma~\ref{Lemma: affine translations} there exist $t\in W_{J'}$ with $W_{J'} = \conv(\{t^z \mid z\in \ZZ\})$ and $r \geq 1$ with $[t,(w')^r]=1$ in $W$. As $J \setminus J' \subseteq (J')^{\perp}$, we have $[t,w^r]=1$. In particular, $h_z := [\widetilde{w}^r, \widetilde{t}^z] \in N \cap \Fix(c) = T$ for all $z\in \ZZ$, hence $\widetilde{t}^{-z} \widetilde{w}^r \widetilde{t}^z = \widetilde{w}^r h_z$. As $\Fix_H(\Sigma) \leq H_c$ is compact, Lemma~\ref{Lemma: Generalization of Lemma 3.5 in BW04} implies $\con(\widetilde{w}^r) = \con(\widetilde{w}^r h_z)$ for all $z\in\ZZ$. As $T \trianglelefteq N$, we have $h_z' := \widetilde{w}^r h_z\inv \widetilde{w}^{-r} \in T$ and $(\widetilde{w}^r h_z)\inv = \widetilde{w}^{-r} h_z'$. Applying Lemma~\ref{Lemma: Generalization of Lemma 3.5 in BW04} again, we infer $\nub(\widetilde{w}^r) = \nub(\widetilde{w}^r h_z)$.

	On the other hand, $\widetilde{t}^{-z} H_c \widetilde{t}^z$ is tidy for $\widetilde{t}^{-z} \widetilde{w}^r \widetilde{t}^z = \widetilde{w}^r h_z$ for all $z\in\ZZ$ by Lemmas~\ref{Lemma: auxiliary results for tidy subgroups}(3) and \ref{lemma: Stabilizer is tidy}. As $\nub(\widetilde{w}) = \nub(\widetilde{w}^r) = \nub(\widetilde{w}^r h_z)$ and $\Fix_H( \Sigma \cap R_{J'}(c)  )=\bigcap_{z\in\ZZ}\widetilde{t}^zH_c\widetilde{t}^{-z}$, the claim follows.
\end{proof}

\begin{theorem}\label{Theorem: Upper bound for nub}
	Let $w\in W$ be straight and standard, and let $J = \supp(w)$. Then:
	\[ \nub(\widetilde{w}) \leq \left( \bigcap_{z\in \ZZ} \widetilde{w}^z V \widetilde{w}^{-z} \right) \cap \Fix_H( R_{J^{\perp} \cup \Jsph^w}(c) ) \cap \Fix_H( \Sigma \cap R_{\JaffCox}(c) ).  \]
\end{theorem}
\begin{proof}
	Note that $W_{\Jsph^w}$ is spherical and normalized by $w$. The claim thus follows from Propositions~\ref{prop: Upper bound for nub} and \ref{Proposition: nub contained in fixators}, since $\Fix( R_{\Jperp}(c) ) \cap \Fix( R_{\Jsph^w}(c) ) \subseteq \Fix( R_{\Jperp \cup \Jsph^w}(c) )$ by Lemma~\ref{Lemma: Fixator of orthogonal residues}.
\end{proof}

\subsection{Contraction of root groups}\label{subsection:contractionrootgroups}

We next provide a lower bound for the nub of an element $\widetilde{w}$ with $w\in W$ straight and standard.

\begin{lemma}\label{Lemma: root groups in the nub}
	Let $w\in W$ be straight and standard, and let $J = \supp(w)$. Let $\alpha \in \Phi$ be a root. Then one of the following holds:
	\begin{enumerate}
		\item $\supp(\alpha) \subseteq J^{\perp} \cup \Jsph^w \cup \JaffCox$;
		
		\item $w^n\alpha\neq\alpha$ for all $n\in\NN^*$.
	\end{enumerate}
\end{lemma}
\begin{proof}
	Suppose $\supp(\alpha) \not\subseteq J^{\perp} \cup \Jsph^w \cup \JaffCox$. Assume for a contradiction that $w^n \alpha = \alpha$ for some $n \in \NN^*$. Then $r_{\alpha}$ commutes with $w^n$ and hence $r_{\alpha}$ normalizes $\Pc(w^n)$. As $w$ is straight and $J = \supp(w)$, the element $w^n$ is straight and $J = \supp(w^n)$. Now Lemma~\ref{Lemma: straight element J essental} yields that $\Pc(w^n) = W_J$ and that $J$ is essential. Lemma~\ref{lemma:prelim_essential} then implies that $r_{\alpha} \in N_W(W_J) = W_J \times W_{J^{\perp}}$. Hence $\supp(\alpha) \subseteq J \cup J^{\perp}$. Recall that by Lemma~\ref{Lemma: support of a root irreducible} the support of a root is irreducible. Using our assumption, we have $\supp(\alpha) \subseteq J$.  Let $J_\alpha \subseteq J$ be the component of $J$ with $\supp(\alpha) \subseteq J_\alpha$. Again, our assumption implies that $J_\alpha$ is non-affine. As $J$ is essential, $J_\alpha$ is indefinite. Let $w' \in W_{J_\alpha}$ and $w'' \in W_{J \setminus J_\alpha}$ with $w = w' w''$. Note that $\alpha = w^n \alpha = (w')^n (w'')^n \alpha = (w')^n \alpha$. As $w$ is standard, it follows that $w'$ is standard (cf.\ \S\ref{Subsection: Standard elements}). Let $P_{w'}^{\max}=P_{(w')^n}^{\max} = W_{I}$ ($I \subseteq J_\alpha$) be the largest spherical parabolic subgroup inside $W_{J_\alpha}$ which is normalized by $(w')^n$. As $(w')^n \alpha = \alpha$, the spherical parabolic subgroup $\langle r_{\alpha} \rangle$ is normalized by $(w')^n$. Hence $r_{\alpha} \in W_{I} \leq W_{\Jsph^w}$, as $I \subseteq \Jsph^w$. But $\supp(\alpha) \not\subseteq \Jsph^w$, which is a contradiction.
\end{proof}

\begin{definition}\label{definition:FPRS}
	We say that the completion $H$ of $\mc{D}$ has the \emph{(FPRS) property\footnote{FPRS stands for ``Fixed Points of Root Subgroups''}} if for any sequence of roots $(\alpha_n)_{n\geq 0}$ of $\Phi$ such that  $\lim_{n\to +\infty} \dist(1_W,\alpha_n)=+\infty$, we have $U_{-\alpha_n}\stackrel{n\to\infty}{\to}1$, i.e.\ for each open neighborhood $U$ of $1_H$ there exists $N_U \in \NN$ such that for all $n \geq N_U$ we have $U_{-\alpha_n} \subseteq U$.
\end{definition}

\begin{remark}
	Note that if $H$ is the geometric completion of the RGD system $\mc{D}$, then $H$ has the (FPRS) property if and only if $\mc{D}$ has the (FPRS) property in the sense of \cite[\S$2.1$]{CR09} (where this property was first introduced).
\end{remark}

\begin{remark}
	Property (FPRS) for RGD systems is a natural condition which is satisfied in a lot of cases. Let us mention the examples of RGD systems of Kac--Moody groups, those constructed in \cite{RR06} and almost all $2$-spherical RGD systems, where an RGD system of type $(W, S)$ is \emph{$2$-spherical}, if $m_{st} < \infty$ for all $s, t \in S$. We refer to \cite[Proposition~$4$]{CR09} for more details. However, there are examples of RGD systems which do not satisfy Property (FPRS) (cf.\ \cite[Remark before Lemma~$5$]{CR09} and \cite[Corollary~B]{BiRGD}, and also \cite[Theorem~G]{BiUncountablymanyRGD-systems} for the $2$-spherical case).
\end{remark}

\begin{proposition}\label{prop: root groups in the nub}
Assume that $H$ has the (FPRS) property. Let $w\in W$ be straight and standard, and let $J = \supp(w)$. Let $\alpha \in \Phi$ be such that $\{ w^z \mid z\in \ZZ \} \subseteq \alpha$. If $\supp(\alpha) \not\subseteq \Jperp \cup \Jsph^w \cup \JaffCox$, then  $U_{\alpha} \leq \con(\widetilde{w}) \cap \con(\widetilde{w}\inv) \leq \nub(\widetilde{w})$.
\end{proposition}
\begin{proof}
Let $\epsilon\in\{\pm 1\}$. By Lemma~\ref{Lemma: root groups in the nub},  $w^{\epsilon n}\alpha\neq\alpha$ for all $n\in\NN^*$. As $w^{\epsilon n}\alpha\in\Phi_+$ for all $n\in\NN$ by assumption and as there are only finitely many positive roots $\beta$ such that $\dist(1_W,-\beta)\leq N$ for any given $N\in\NN$, we have $\lim_{n\to +\infty} \dist(1_W,-w^{\epsilon n}\alpha)=+\infty$. Hence $\widetilde{w}^{\epsilon n}U_{\alpha}\widetilde{w}^{-\epsilon n}=U_{w^{\epsilon n}\alpha}\to 1$ as $n\to\infty$ by the (FPRS) property, that is, $U_{\alpha}\leq\con(\widetilde{w}^{\epsilon})$.
\end{proof}

\section{The nub in maximal Kac--Moody groups}\label{Section: The nub in Kac--Moody groups}

In this section, we determine precisely the nub in $H$ of an element $\widetilde{w}$ with $w\in W$ straight when $H$ is the scheme-theoretic completion of a (minimal) Kac--Moody group over a finite field. We start with some preliminaries on Kac--Moody algebras and groups. General references for Kac--Moody algebras and groups are \cite{Kac90} and \cite{Ma18}.

\subsection{Preliminaries on Kac--Moody algebras and groups}\label{Section: Maximal Kac--Moody groups}

\subsection*{Generalized Cartan matrices}\label{subsubsection:GCM}

Throughout, we fix a generalized Cartan matrix (GCM) $A=(a_{ij})_{i,j\in I}$, as well as a realization $(\hh,\Pi,\Pi^{\vee})$ of $A$ in the sense of \cite[\S 1.1]{Kac90}, with set of \emph{simple roots} $\Pi=\{\alpha_i \ | \ i\in I\}$ and set of \emph{simple coroots} $\Pi^{\vee}=\{\alpha^{\vee}_i \ | \ i\in I\}$.

The \emph{Weyl group} $W = W(A)$ of $A$ is the subgroup of $\GL(\hh^*)$ generated by the \emph{simple reflections} $s_i$ ($i\in I$) defined by
$$s_i\co\hh^*\to\hh^*: \alpha\mapsto \alpha-\la\alpha,\alpha_i^{\vee}\ra\alpha_i.$$
The pair $(W, S:=\{s_i \ | \ i\in I\})$ is then a Coxeter system. When convenient, we will  identify a subset $J\subseteq I$ with $\{s_i \ | \ i\in J\}\subseteq S$ (writing for instance $W_J$ for the parabolic subgroup generated by $\{s_i \ | \ i\in J\}$).

We call $A$ \emph{indecomposable} if there is no proper nonempty $J\subseteq I$ with $a_{ij}=0$ for all $i\in I\setminus J$ and $j\in J$, or equivalently if $(W, S)$ is irreducible. An indecomposable GCM $A$ can be of \emph{finite}, \emph{affine} or \emph{indefinite} type (see \cite[Chapter~4]{Kac90} or Remark~\ref{remark:JaffvsJaffCox} below). Accordingly, we call a subset $J\subseteq I$ of \emph{finite}, \emph{affine} or \emph{indefinite} type if the GCM $A_J:=(a_{ij})_{i,j\in J}$ is of that type. As before, we define the \emph{components} of $J\subseteq I$ as the subsets $J_1,\ldots,J_r\subseteq J$ such that $W_J=W_{J_1}\times\cdots\times W_{J_r}$ with $J_i$ irreducible for each $i$. Given $J\subseteq I$, we denote by $\Jaff$ the union of all the components of $J$ of affine type.

\begin{remark}\label{remark:JaffvsJaffCox}
	The GCM $A$ is of finite/affine/indefinite type if and only if $(W,S)$ is of that type (in the sense of \S\ref{subsection:CoxSyst}), with one exception: when $A=\begin{psmallmatrix}2&a_{12}\\ a_{21}&2\end{psmallmatrix}$ with $a_{12}a_{21}>4$, then $A$ is of indefinite type, while $W$ is the infinite dihedral group, hence of affine type. Thus, for a subset $J\subseteq I$, we always have $\Jaff\subseteq\JaffCox$ but the inclusion might be proper.
\end{remark}

\subsection*{Kac--Moody root systems}\label{subsubsection:KMRS}
The \emph{Kac--Moody algebra} $\g(A)$ associated with $A$ admits a root space decomposition $\g(A)=\hh\oplus\bigoplus_{\alpha\in\Delta}\g_{\alpha}$ with respect to the adjoint action of the Cartan subalgebra $\hh$, with corresponding root spaces $$\g_{\alpha}:=\{x\in\g(A) \ | \ [h,x]=\alpha(h)x \ \forall h\in\hh\}$$ and root system $\Delta:=\{\alpha\in\hh^*\setminus\{0\} \ | \ \g_{\alpha}\neq\{0\}\}$. We also set $\g_A:=[\g(A),\g(A)]=\hh'\oplus \bigoplus_{\alpha\in\Delta}\g_{\alpha}$, where $\hh':=\sum_{i\in I}\mathbb{C}\alpha_i^{\vee}\subseteq\hh$.

Set $Q_+:=\bigoplus_{i\in I}\NN\alpha_i$. We let $\Delta^+:=\Delta\cap Q_+$ denote the set of \emph{positive roots}, so that $\Delta=\Delta^+\dot{\cup}(-\Delta^+)$. The \emph{height} of a root $\alpha=\pm\sum_{i\in I}n_i\alpha_i$ is $\height(\alpha):=\pm\sum_{i\in I}n_i\in\ZZ$. We also define the \emph{support} 
$$\supp(\alpha):=\{i\in I \ | \ n_i\neq 0\}\subseteq I$$
of an element $\alpha \in \pm Q_+$ (see also Remark~\ref{Remark: Two notions of support}). The support of a root is always connected. For a subset $J\subseteq I$, we set
\begin{equation}\label{eqn:Delta+K}
\Delta(J):=\Delta\cap \bigoplus_{i\in J}\ZZ\alpha_i, \quad \Delta^+(J):=\Delta^+\cap\Delta(J), \quad \textrm{and}\quad \Delta^+_J:=\Delta^+\setminus\Delta^+(J).
\end{equation}

The Weyl group $W$ stabilizes $\Delta$. One defines the sets of \emph{real} and \emph{imaginary} roots, respectively, as $$\Delta^{\re}:=W.\{\alpha_i \ | \ i\in I\}\subseteq\Delta\quad\textrm{and}\quad \Delta^{\im}:=\Delta\setminus\Delta^{\re}.$$
One also sets $\Delta^{\rep}:=\Delta^{\re}\cap\Delta^+$ and $\Delta^{\imp}:=\Delta^{\im}\cap\Delta^+$, as well as $\Delta^{\imp}(J) := \Delta^{\imp} \cap \Delta(J)$ for each $J \subseteq I$. 

\begin{remark}\label{Remark: Two notions of support}
Set $\Phi:=\Phi(W,S)$. There is a $W$-equivariant bijection $\Delta^{\re}\to\Phi:\alpha_i\mapsto \alpha_{s_i}$ mapping $\Delta^{\rep}$ onto $\Phi_+$. Identifying $\Delta^{\re}$ and $\Phi$, the two notions of support of a (real) root that we defined then coincide (cf.\ \cite[Exercise~$6.3$]{Ma18} and Lemma~\ref{Lemma: support of a root irreducible}).
\end{remark}

The set of positive imaginary roots can be described as $\Delta^{\imp}=\bigcup_{w\in W}wK_0$, where $$K_0:=\{\alpha\in Q_+ \setminus \{0\} \ | \ \textrm{$\supp(\alpha)$ is connected and} \ \la\alpha,\alpha_i^{\vee}\ra\leq 0 \ \forall i\in I\}.$$
Moreover, for any $\alpha\in\Delta^{\imp}$, there is a unique root $\alpha_{\min}\in W\alpha\cap K_0$, namely, $\alpha_{\min}$ is the unique element of minimal height in $W\alpha$. If $\alpha\in\Delta^{\imp}$, then $\supp(\alpha)$ is of affine or indefinite type. If $\alpha\in K_0$, then $\supp(\alpha)$ is of affine type if and only if
\begin{equation}\label{eqn:prelim:affine_roots}
	\la\alpha,\alpha_i^{\vee}\ra=0 \quad\textrm{for all $i\in\supp(\alpha)$.}
\end{equation}

\subsection*{Kac--Moody groups}\label{subsubsection:KMG}

Let $k=\mathbb{F}_q$ be a finite field of order $q$ and characteristic $p$. Let $\G_A$ be the constructive Tits functor of simply connected type associated to $A$ introduced by Tits (\cite{Ti87}), and let $G:=\G_{A}(k)$ be the corresponding \emph{minimal Kac--Moody group} over $k$. Thus $G$ is an amalgamated product of a (finite) \emph{torus} $T:=\T(k)$ exponentiating $\hh'$, and of the \emph{real root groups} $U_{\alpha}:=\U_{\alpha}(k)\cong (k,+)$ ($\alpha\in\Delta^{\re}$) exponentiating the real root spaces $\g_{\alpha}$. Then $\mc{D}_A := ( G, (U_{\alpha})_{\alpha \in \Delta^{\re}} )$ is an RGD system of type $(W,S)$, and $T=\bigcap_{\alpha\in\Delta^{\re}}N_G(U_{\alpha})$. We keep the notations $U_{\pm}=\langle U_{\alpha} \ | \ \alpha\in\Delta^{\mathrm{re}\pm}\rangle$ and $N$ from \S\ref{subsection:RGDsystems}, and denote as before by $\mc{B}=\mc{B}(\mc{D}_A)$ the associated building, with fundamental chamber $c$ and fundamental apartment $\Sigma\approx\Sigma(W,S)$. The kernel of the action map $\rho\co G\to\Aut_0(\mc{B})$ coincides with the center $\mathcal{Z}(G)\subseteq T$ of $G$.

We also consider the \emph{maximal Kac--Moody group} $G^{\pma}:=\G_{A}^{\pma}(k)$ of simply connected type associated to $A$, which coincides with the scheme-theoretic completion of $G$ (see \cite[\S 8.5]{Ma18}). Thus $G^{\pma}$ is a first-countable tdlc group constructed as an amalgamated product of $G$ with a certain completion $U^{\map}:=\U^{\map}_A(k)$ of $U_+$. The group $U^{\map}$ is topologically generated by the real root groups $U_{\alpha}$ ($\alpha\in\Delta^{\rep}$), together with the \emph{imaginary root groups} $U_{\beta}$ ($\beta\in\Delta^{\imp}$) exponentiating the subalgebras $\oplus_{n\in\NN^*}\g_{n\beta}$. 

More precisely, let $\g_{\ZZ}=\hh'_{\ZZ}\oplus\bigoplus_{\alpha\in\Delta}\g_{\alpha\ZZ}$ be the $\ZZ$-form of $\g_A$ defined in \cite[Definition~7.5]{Ma18}, and fix a $\ZZ$-basis $\mc{B}_{\ZZ}$ of $\n^+_{\ZZ}=\bigoplus_{\alpha\in\Delta^+}\g_{\alpha\ZZ}$  consisting of homogeneous elements, that is, of elements $x\in\g_{\alpha\ZZ}$ for some $\alpha\in\Delta^+$ (in which case we set $\deg(x):=\alpha$). Choose a total order on $\mc{B}_{\ZZ}$. Then each element $g$ of $U^{\map}$ can be written in a unique way as a product of ``twisted exponentials'' $[\exp](\lambda_xx)$ with $\lambda_x\in k$ and $x\in \mc{B}_{\ZZ}$ (see \cite[Theorem~8.51]{Ma18}):
$$g=\prod_{x\in\mc{B}_{\ZZ}}[\exp](\lambda_xx)\quad\textrm{for some $\lambda_x\in k$},$$
where the product is taken in the given order on $\mc{B}_{\ZZ}$ (this is called the \emph{standard form} of $g$).

More generally, if $\Psi\subseteq\Delta^+$ is a \emph{closed} set of roots (i.e.\ such that $\alpha+\beta\in\Psi$ whenever $\alpha,\beta\in\Psi$ and $\alpha+\beta\in\Delta$), then $\bigoplus_{\alpha\in\Psi}\g_{\alpha\ZZ}$ is a subalgebra of $\n^+_{\ZZ}$ with $\ZZ$-basis $\mc{B}_{\Psi}:=\{x\in \mc{B}_{\ZZ} \ | \ \deg(x)\in\Psi\}$, and the subset $U^{\ma}_{\Psi} := \U^{\ma}_{\Psi}(k)$ of $U^{\map}$ consisting of those elements $g$ of the form $g=\prod_{x\in\mc{B}_{\Psi}}[\exp](\lambda_xx)$ for some $\lambda_x\in k$ is a closed subgroup of $U^{\map}$. For instance, if $\beta\in\Delta^{\imp}$, then $\Psi=\NN^*\delta$ is closed and we set $U_{\beta}:=U^{\ma}_{\NN^*\beta}$, while if $\alpha\in \Delta^{\rep}$, then $\Psi=\{\alpha\}$ is closed and we have $U_{\alpha}=U^{\ma}_{\{\alpha\}}$. In particular, for any closed set $\Psi \subseteq \Delta^+$, the group $U^{\ma}_{\Psi}$ is topologically generated by the root groups $U_{\alpha}$ with $\alpha \in \Psi$. Note also that for $\Psi_1, \Psi_2 \subseteq \Delta^+$ closed we have $U^{\ma}_{\Psi_1} \cap U^{\ma}_{\Psi_2} = U^{\ma}_{\Psi_1 \cap \Psi_2}$.

\begin{lemma}[{\cite[Lemma~$8.58(2)$]{Ma18}}]\label{Lemma: Unique decomposition}
	Let $\Psi' \subseteq \Psi \subseteq \Delta^+$ be closed sets of roots. If $\Psi \setminus \Psi'$ is closed, then we have a unique decomposition $U^{\ma}_{\Psi} = U^{\ma}_{\Psi'} \cdot U^{\ma}_{\Psi \setminus \Psi'}$.
\end{lemma}

The group $U^{\map}$ is compact open in $G^{\pma}$, and its normal subgroups $U^{\ma}_n:=\U^{\ma}_{\Psi(n)}(k)$ ($n\in\NN$), where $\Psi(n):=\{\alpha\in\Delta^+ \ | \ \height(\alpha)\geq n\}$, form a basis of identity neighbourhoods in $G^{\pma}$. The torus $T$ normalizes each root group $U_{\alpha}$ ($\alpha\in\Delta^{\re}\cup\Delta^{\imp}$). The couple $(B^{\map}:=TU^{\map},N)$ is a BN-pair for $G^{\pma}$, with same associated building $\mc{B}$. The action map $G^{\pma}\to\Aut_0(\mc{B})$ is continuous and restricts to $\rho\co G\to\Aut_0(\mc{B})$.

The closure $\overline{G}$ of $G$ in $G^{\pma}$ coincides with $G^{\pma}$ 
as soon as 
$$p>M_A:=\max_{i\neq j}|a_{ij}|,$$
but may be properly contained in $G^{\pma}$ otherwise.

\begin{remark}\label{remark: schematic-completion is Kac--Moody completion}
	The groups $G_A^{\pma}$ and $\overline{G}$ are completions of the RGD system $\mc{D}_A$ in the sense of Section~\ref{section:completionsRGD}. Moreover, $U^{\map}$ is an admissible subgroup of $G^{\pma}$ (see Lemma~\ref{lemma:criterion_unipotent} with $N_s = U^{\ma}_{\Delta^+ \setminus \{ \alpha_s \}}$ and Remark~\ref{Remark: Two notions of support}). Similarly, the closure $\overline{U_+}$ of $U_+$ in $U^{\map}$ is an admissible subgroup of $\overline{G}$ (see Lemma~\ref{lemma:criterion_unipotent}). 
	Note also that both $G^{\pma}$ and $\overline{G}$ have the (FPRS) property from Definition~\ref{definition:FPRS}, by definition of the topology on $G^{\pma}$. 
\end{remark}

\begin{remark}\label{remark:wtildeconjugationaction}
	We fix a section $W \to N: w \mapsto \widetilde{w}$ of the homomorphism $N \to W$ as in \cite[Definition~$7.58$]{Ma18}: for each $i\in I$, it maps $s_i$ to the element $\widetilde{s}_i\in U_{\alpha_i}U_{-\alpha_i}U_{\alpha_i}$ from \cite[Equation~$7.34$ on page $135$]{Ma18}, and is such that $\widetilde{w} := \widetilde{s}_{i_1} \cdots \widetilde{s}_{i_r}$ is independent of the choice of reduced decomposition $w = s_{i_1} \cdots s_{i_r}$ for $w\in W$. The conjugation action of $N=\langle \{\widetilde{s}_i \ | \ i\in I\}\cup T\rangle\leq G$ on root groups is then compatible with the $W$-action on roots:
	\begin{equation}\label{eqn:wUalphawinv}
		\widetilde{w}U_{\alpha}\widetilde{w}\inv=U_{w\alpha}\quad\textrm{for all $w\in W$ and $\alpha\in \Delta^{\re}\cup\Delta^{\imp}$.}
	\end{equation}
	More precisely, $N$ admits an adjoint action $\Ad_k\co N\to\Aut(\g_k):\widetilde{w}\mapsto w^*$ on $\g_k:=\g_{\ZZ}\otimes_{\ZZ}k$, such that
	\begin{equation}
		w^*\g_{\alpha k}=\g_{w\alpha k}\quad\textrm{for all $w\in W$ and $\alpha\in \Delta$,}
	\end{equation}
	where $\g_{\alpha k}:=\g_{\alpha\ZZ}\otimes_{\ZZ}k$. Moreover, if $\alpha\in\Delta^+$ and $x\in \mc{B}_{\ZZ}\cap \g_{\alpha\ZZ}$, then for all $\lambda_x\in k$ we have
	\begin{equation}
	\widetilde{w}[\exp](\lambda_xx)\widetilde{w}\inv=[\exp](\lambda_xw^*x)\quad\textrm{for all $w\in W$}
	\end{equation}
	(note that if $\alpha\in\Delta^+$ is such that $w\alpha\in\Delta^-$, then necessarily $\alpha\in \Delta^{\rep}$ and the notation $[\exp](\lambda_xw^*x)\in U_{w\alpha}$ still makes sense for $x\in\g_{\alpha\ZZ}$, as in this case $k\to U_{w\alpha}:\lambda\mapsto [\exp](\lambda w^*x)$ is just a parametrisation of $U_{w\alpha}$ as in \cite[\S7.4.3 equation (7.26)]{Ma18}). 
	
	Finally, note that if $K \subseteq I$, then $U^{\ma}_{\Delta^+_K}$ is normalized by $\widetilde{w}$ for any $w\in W_K$ by (\ref{eqn:wUalphawinv}). As the $K$-residue $R_K(c)$ of $\mathcal B$ is given by $R_K(c) = U^{\ma}_{\Delta^+(K)}W_K c$, and as $U^{\ma}_{\Delta^+_K}$ is normalized by $U^{\ma}_{\Delta^+(K)}$  (see \cite[Lemma~$8.58(4)$]{Ma18}), we deduce that
	\begin{equation}\label{equation:UmaDelta+subK fixes R_K(c)}
		U^{\ma}_{\Delta^+_K} \leq \Fix(R_K(c)).
	\end{equation}
\end{remark}

\begin{example}\label{example:Deltawdecomposition}
For $w\in W$, the sets $\Delta_w:=\{\alpha\in\Delta^+ \ | \ w\inv \alpha\in\Delta^-\}$ and $\Delta^+\setminus\Delta_w$ are both closed. In particular, there is a unique decomposition $U^{\map}=U^{\ma}_{\Delta_w}\cdot U^{\ma}_{\Delta^+\setminus\Delta_w}$ by Lemma~\ref{Lemma: Unique decomposition}. In view of Remark~\ref{remark:wtildeconjugationaction} and the fact that $U_- \cap U^{\map} = \{1\}$ (see \cite[Theorem~8.78 and Definition~B.30 axiom (RT3)]{Ma18}), we then have
\begin{equation}
U^{\ma}_{\Delta^+\setminus\Delta_w}=U^{\map}\cap \widetilde{w} U^{\map}\widetilde{w}\inv.
\end{equation}
\end{example}

\subsection{Lower bound for the nub}
We start by obtaining a lower bound for the nub of an element $\widetilde{w}$ with $w\in W$ straight, by strengthening the results of \S\ref{subsection:contractionrootgroups}, taking this time into account imaginary root groups as well.

Since $W$ is a finitely generated linear group, there exists by Selberg's lemma a torsion-free finite index normal subgroup $W_0$ of $W$, which we fix in the sequel. For $w\in W$ we define the \emph{essential support} of $w$ as
\[ \esupp(w) := \bigcap_{n\in \NN} \supp(w^n) \subseteq S. \]
We recall that a subset $J \subseteq S$ is \emph{essential} if it is nonempty and has no spherical component.

\begin{lemma}\label{lemma:essssup}
	Let $w_0\in W_0$ be straight, and let $w\in W$ be conjugate to $w_0$. Then:
	\begin{enumerate}
		\item $\supp(w_0)\subseteq \esupp(w)$.
		
		\item $\esupp(w)$ is essential and there is some $N\in\NN^*$ such that $\esupp(w)=\supp(w^{Nn})$ for all $n\in\NN^*$.
	\end{enumerate}
\end{lemma}
\begin{proof}
	(1) Let $n\in\NN$ and $w\in W$ be conjugate to $w_0$. Note that $w_0^n$ is straight, hence cyclically reduced. In the terminology of \cite{Ma14b}, it then follows from \cite[Corollary~B and Remark~1]{Ma14b} that a reduced decomposition of $w_0^n$ can be obtained from a reduced decomposition of $w^n$ by performing cyclic shifts, braid relations and $ss$-cancellations. In particular, $\supp(w_0)=\supp(w_0^n)\subseteq\supp(w^n)$, proving $(a)$.
	
	(2) Let $M\in\NN^*$ such that $\esupp(w)=\bigcap_{k=1}^M\supp(w^k)$ and set $N:=M!$. Then $\supp(w^{Nn})\subseteq\supp(w^N)\subseteq \esupp(w)$ and hence $\esupp(w)=\supp(w^{Nn})$ for all $n\in\NN^*$. Assume now for a contradiction that $ \esupp(w)=\supp(w^{N})$ has a spherical component $J$, and let $N_J$ be the order of $W_J$. Then $\esupp(w)=\supp(w^{N_JN})$ intersects $J$ trivially, a contradiction.
\end{proof}

\begin{lemma}\label{lemma:contrim}
	Let $w_0\in W_0$ be straight and let $J = \supp(w_0)$. Let $\alpha\in\Delta^{\imp}$. Let $v\in W$ such that $\alpha=v\alpha_{\min}$, and set $J_v:=\esupp(v\inv w_0v)$. Then $J_v\supseteq J$, and one of the following holds:
	\begin{enumerate}
		\item
		$w_0^n\alpha\neq\alpha$ for all $n\in\NN^*$.
		\item
		$\supp(\alpha_{\min})\subseteq J_v^\perp\subseteq J^\perp$.
		\item
		$\supp(\alpha_{\min})$ is an affine component of $J_v$. In addition, $\supp(\alpha_{\min})$ is either a component of $J$ or contained in $J^\perp$.
	\end{enumerate}
\end{lemma}
\begin{proof}
	Let $\tilde{J}\in \{J,J_v\}$. Note first that $J_v\supseteq J$ by Lemma~\ref{lemma:essssup}(1).  Hence also $J_v^\perp\subseteq J^\perp$. We abbreviate $\beta := \alpha_{\min}$. Assume that $w_0^n\alpha=\alpha$ for some $n\in\NN^*$ and that $\supp(\beta)\not\subseteq \tilde{J}^\perp$ (hence $\supp(\beta)\not\subseteq J_v^\perp$). Set $w:=v\inv w_0v$, so that our assumptions can be rewritten as $w^n\beta=\beta$ and $\supp(\beta)\not\subseteq \tilde{J}^\perp$.
	
	Let $w^n=s_{i_1}\cdots s_{i_d}$ be a reduced decomposition of $w^n$. For each $t\in\{1,\ldots,d\}$, set $\beta_t:=s_{i_1}\cdots s_{i_{t-1}}\alpha_{i_t}\in\Delta^{\rep}$, so that $$\{\beta_t \ | \ 1\leq t\leq d\}=\Delta^+\cap w^n\Delta^-$$ (see e.g.\ \cite[Exercise~4.33]{Ma18}). Then $$0=\beta-w^n\beta=\sum_{t=1}^d\la\beta,\alpha_{i_t}^{\vee}\ra \beta_t$$ (see e.g.\ \cite[Exercise~4.34]{Ma18}). 
	Hence $\la\beta,\alpha_{i}^{\vee}\ra=0$ for all $i\in\supp(w^n)$ since $\la \beta, \alpha_i^\vee \ra \leq 0$ (because $\beta \in K_0$) and $\beta_t \in Q_+$ for all $t = 1, \ldots, d$. In particular, as $J_v = \esupp(w) \subseteq \supp(w^n)$, we deduce that
	\begin{equation}\label{eqn:aaivee}
		\la\beta,\alpha_{i}^{\vee}\ra=0\quad\textrm{for all $i\in J_v$ (hence for all $i \in \tilde{J}$).}
	\end{equation}
	
	Note that $J_{\beta}:=\tilde{J}\cap\supp(\beta)$ is nonempty: otherwise, writing $\beta = \sum_{j\in \supp(\beta)} n_j \alpha_j$ with $n_j \in \NN^*$, (\ref{eqn:aaivee}) would yield $\sum_{j\in \supp(\beta)} n_j a_{ij} = \la \beta, \alpha_i^\vee \ra = 0$ with $a_{ij} \leq 0$ for all $i \in \tilde{J}$, and hence $\supp(\beta)\subseteq \tilde{J}^\perp$, a contradiction. We can thus write $$\beta=\beta_J+\beta'$$ for some nonzero $\beta_J\in Q_+$ with $\supp(\beta_J) = J_{\beta}$ and some $\beta'\in Q_+$ with $\supp(\beta')\cap \tilde{J}=\varnothing$. Note that $\la \beta',\alpha_i^{\vee}\ra\leq 0$ for all $i\in \tilde{J}$ as $a_{ij}\leq 0$ whenever $i\neq j$. Hence (\ref{eqn:aaivee}) implies that
	\begin{equation}\label{eqn:aJivee}
		\la \beta_J,\alpha_i^{\vee}\ra\geq 0\quad\textrm{for all $i\in \tilde{J}$.}
	\end{equation}
	In particular, as $\la \beta_J, \alpha_i^{\vee} \ra \leq 0$ for all $i \in \tilde{J} \setminus J_{\beta}$, we infer that
	\begin{equation}\label{eqn:aJiveebis}
		\la \beta_J,\alpha_i^{\vee}\ra= 0\quad\textrm{for all $i\in \tilde{J}\setminus J_{\beta}$.}
	\end{equation}
	On the other hand, (\ref{eqn:aJivee}) and \cite[Theorem~4.3]{Kac90} imply that either $J_{\beta}$ is of spherical type, or 
	\begin{equation}\label{eqn:aJiveeter}
		\la\beta_J,\alpha_i^{\vee}\ra=0\quad\textrm{for all $i\in J_{\beta}$.}
	\end{equation}
	But if $J_{\beta}$ is spherical, then as $\tilde{J}$ has no spherical components (see Lemma~\ref{Lemma: straight element J essental} and Lemma~\ref{lemma:essssup}$(2)$), there exists some $i\in \tilde{J}\setminus J_{\beta}$ such that $\la\beta_J,\alpha_i^{\vee}\ra<0$, contradicting (\ref{eqn:aJiveebis}). Hence 
	\begin{equation}\label{eqn:aprimeaivee}
		\la \beta',\alpha_i^{\vee}\ra=\la \beta_J,\alpha_i^{\vee}\ra= 0\quad\textrm{for all $i\in \tilde{J}$}
	\end{equation}
	by (\ref{eqn:aaivee}), (\ref{eqn:aJiveebis}) and (\ref{eqn:aJiveeter}). In particular, $\supp(\beta')\subseteq \tilde{J}^\perp$, and hence $\beta'=0$ because $\supp(\beta)$ is connected and $\supp(\beta_J) \subseteq \tilde{J}$. Moreover, $\supp(\beta)=\supp(\beta_J)\subseteq \tilde{J}$ is of affine type by \cite[Theorem~4.3]{Kac90}. Finally, (\ref{eqn:aaivee}) implies that $\tilde{J}':=\tilde{J}\setminus \supp(\beta)$ is contained in $\supp(\beta)^\perp$. In particular, $\supp(\beta)$ is an affine component of $\tilde{J}$.
	
	The case $\tilde{J} = J_v$ thus implies that if (1) and (2) do not hold, then $\supp(\beta)$ is an affine component of $J_v$. Moreover, the case $\tilde{J} = J$ implies that if (1) does not hold, then $\supp(\beta)$ is either contained in $\Jperp$ or is an affine component of $J$. Thus, if (1) and (2) do not hold, then (3) holds, as desired.
\end{proof}

\begin{theorem}\label{thm: root groups are contracted}
	Let $w \in W$ be straight and let $J = \supp(w)$. Let $\alpha\in\Delta^{\imp}$. Then exactly one of the following holds:
	\begin{enumerate}
		\item $w^n\alpha\neq\alpha$ for all $n\in\NN^*$.
		
		\item $\supp(\alpha)\subseteq J^\perp$.
		
		\item $\alpha\in K_0$ and $\supp(\alpha)$ is an affine component of $J$.
	\end{enumerate}
	In particular, either $w^n\alpha\neq\alpha$ for all $n\in\NN^*$ or $W_J \alpha=\{\alpha\}$.
\end{theorem}
\begin{proof}
Note that the conditions (1), (2), (3) are mutually exclusive, as $W_J \alpha=\{\alpha\}$ in cases (2) and (3), and as $J\cap J^{\perp}=\varnothing$.

	Set $\beta := \alpha_{\min}$. Let $v\in W$ be of minimal length such that $\alpha=v\beta$. Up to replacing $w$ by a power of $w$, we may assume that $w\in W_0$ and that $J_v:=\esupp(w_0)=\supp(w_0)$, where $w_0:=v\inv wv$ (see Lemma~\ref{lemma:essssup}(2)). Assume that $w^n \alpha = \alpha$ for some $n \in \NN^*$ (i.e.\ (1) does not hold), and let us show that (2) or (3) holds.

	Note first that $v$ is the unique element of minimal length in $v W_{J_v}$ (see \cite[Proposition~2.20]{AB08}): by Lemma~\ref{lemma:contrim}, either $\supp(\beta)\subseteq J_v^\perp$ or $\supp(\beta)$ is an affine component of $J_v$. In any case, $ W_{J_v}$ fixes $\beta$. Hence $\alpha=v\beta=vu\beta$ for all $u\in W_{J_v}$, so that $\ell(v)\leq\ell(vu)$ for all $u\in W_{J_v}$.

	Let $R_J:=R_J(1_W)$ and $R_{J_v}:=R_{J_v}(1_W)$ be the standard residues of type $J$ and $J_v$ in the Coxeter building $\Sigma(W,S)$, respectively. Thus $w$ stabilizes $R_J$, and hence $w_0$ stabilizes $v\inv R_J$ and $R_{J_v}$. In particular, $w_0$ stabilizes the projection $R'_J$ of $v\inv R_J$ on $R_{J_v}$ (cf.\ Lemma~\ref{Lemma: projection and isometry commute}). Let $v' := \proj_{R_{J_v}} v\inv \in R_J'$. Then $\ell(v) = \ell( \delta(v\inv, 1_W) ) = \ell( \delta(v\inv, v') ) + \ell( \delta(v', 1_W) )=\ell(vv')+\ell(v')$. Since $v' \in R_{J_v}$ (that is, $v'\in W_{J_v}$) and since $v$ is of minimal length in $v W_{J_v}$, we have $\ell(v)\leq\ell(vv')$ and hence $v' = 1$. Thus $R'_J$ is a standard residue. 

	On the other hand, as $v\inv R_J$ is a residue of minimal rank stabilized by $w_0$ (because $\Pc(w_0)=v\inv  W_J v$, and hence $w_0$ cannot belong to a parabolic subgroup $uW_Ku\inv$ with $|K|<|J|$ by \cite[Lemma~2.25]{AB08}) and as projections do not increase the rank of a residue (see \S\ref{Subsection: Projections and residues}), the residues $R'_J$ and $\proj_{v\inv R_J}R'_J$ (stabilized by $w_0$) also have rank $|J|$. In particular, $R'_J$ and $v\inv R_J$ are parallel by \cite[Proposition~5.37]{AB08}. Moreover, the type $K$ of $R'_J$ is conjugate to $J$ (see \S\ref{Subsection: Projections and residues}). Since $J$ is essential (cf.\ Lemma~\ref{Lemma: straight element J essental}), we then deduce from Lemma~\ref{lemma:prelim_essential} that $K=J$, so that $R_J'=R_J$.

	By \cite[Proposition~21.19(i)]{MPW15}, two residues are parallel if and only if they have the same stabilizer. As $R_J$ and $v\inv R_J$ are parallel residues, they have the same stabilizer $W_J$. Thus $v\in N_{ W}(W_J)=W_J\times W_{J^\perp}$ (see Lemma~\ref{lemma:prelim_essential}). As $ W_J\subseteq W_{J_v}$ by Lemma~\ref{lemma:essssup}(1) and as $v$ has minimal length in $vW_{J_v}$, we conclude that $v\in W_{J^\perp}$. 

	Finally, by Lemma~\ref{lemma:contrim}, either $\supp(\beta)\subseteq J^\perp$, in which case $\supp(\alpha)=\supp(v\beta)\subseteq J^\perp$ (this is case (2)), or $\supp(\beta)$ is an affine component of $J$, in which case $\alpha=\beta$ (this is case (3)).
\end{proof}

Recall that we set $\Delta^+_K:=\Delta^+\setminus\Delta^+(K)$ for $K\subseteq I$ (see (\ref{eqn:Delta+K})).

\begin{corollary}\label{cor: imaginary root groups are contracted}
	Let $w\in W$ be straight and let $J = \supp(w)$. Let $\alpha \in \Delta^{\imp}$ with $\supp(\alpha) \not\subseteq K:=\Jperp \cup \Jaff$. Then $U_{\alpha} \leq \con(\widetilde{w}) \cap \con(\widetilde{w}\inv)$. In particular, $$U^{\ma}_{\Delta^{\imp} \cap \Delta^+_{K}} \leq \nub(\widetilde{w}).$$
\end{corollary}
\begin{proof}
	Let $\alpha \in \Delta^{\imp}$ with $\supp(\alpha) \not\subseteq K$. Then Theorem~\ref{thm: root groups are contracted} implies that $w^n \alpha \neq \alpha$ for all $n\in \NN^*$. In particular, the roots $w^z \alpha$ for $z\in\ZZ$ are pairwise distinct and $w^z \alpha \in \Delta^{\imp}$ for all $z\in \ZZ$. Since there are only finitely many positive roots of any given height, we deduce that $\height(w^z\alpha)\to\infty$ as $z\to \pm\infty$. Hence
	$$\widetilde{w}^z U_{\alpha} \widetilde{w}^{-z}=U_{w^z\alpha} \to 1 \text{ as } z \to \pm \infty,$$
	that is, $U_{\alpha}\subseteq \con(\widetilde{w}) \cap \con(\widetilde{w}\inv) $. In particular,
	\[ U^{\ma}_{\Delta^{\imp} \cap \Delta^+_{K}} = \overline{\la U_{\alpha} \mid \alpha \in \Delta^{\imp}, \, \supp(\alpha) \not\subseteq K \ra} \leq \nub(\widetilde{w}). \qedhere \]
\end{proof}

\begin{corollary}\label{corollary:main_lower_bound_nubKM}
	Let $w\in W$ be straight and standard and let $J = \supp(w)$. Set $L_w := \{ w^zc \mid z\in \ZZ \}\subseteq\mc{B}$ and $K := \Jperp \cup \Jsph^w \cup \Jaff$. Then
	\[ \Fix(L_w) \cap U^{\ma}_{\Delta^+_K}\leq \nub(\widetilde{w}). \]
\end{corollary}
\begin{proof}
	Let $g \in \Fix(L_w) \cap U^{\ma}_{\Delta^+_K}$. By Proposition~\ref{Proposition: Stabilizer by U_+ fixes apartment} (applied to $\Sigma':=g\Sigma$), there exists an element $u \in \overline{\langle U_{\alpha} \mid \alpha \in \Delta^{\rep}, \ L_w \cup ( \Sigma \cap R_K(c) ) \subseteq \alpha \rangle} \leq U^{\ma}_{\Delta^+_K}$ (see (\ref{equation:UmaDelta+subK fixes R_K(c)})) with $ug \in \Fix(\Sigma) \cap U^{\map} = U^{\ma}_{\Delta^{\imp}}$ (cf.\ \cite[Exercise~$8.112(3)$]{Ma18} or \cite[proof of Lemma~5.4]{Ma14}). 
	
	Note that if $\alpha\in\Delta^{\rep}$ satisfies $L_w \cup ( \Sigma \cap R_K(c) ) \subseteq \alpha$, then $\supp(\alpha)\not\subseteq \Jperp \cup \Jsph^w \cup \Jaff$ and $\alpha\supseteq L_w$. We claim that in this case $\supp(\alpha)\not\subseteq \Jperp \cup \Jsph^w \cup \JaffCox$. Otherwise, since $\supp(\alpha)$ is connected, we would have $\supp(\alpha)\subseteq \JaffCox$, and hence $\supp(\alpha)\subseteq J_{\alpha}\subseteq \JaffCox\setminus\Jaff$ for some component $J_{\alpha}$ of $\JaffCox\setminus\Jaff$ (recall that $\Jsph^w \cup \JaffCox\subseteq J$ and that $\Jsph^w\subseteq (\JaffCox)^\perp$). In particular, $|J_{\alpha}|=2$ and $W_{J_{\alpha}}$ is the infinite dihedral group (see Remark~\ref{remark:JaffvsJaffCox}), acting on the simplicial line $\Sigma(W_{J_{\alpha}},J_{\alpha})\subseteq\Sigma$. Write $w=w_{\alpha}w'$ with $w_{\alpha}\in W_{J_{\alpha}}$ and $w'\in W_{J\setminus J_{\alpha}}$. Since $\alpha\supseteq L_w$, the root $w^z\alpha=w_{\alpha}^z\alpha$ contains $c$ for all $z\in\ZZ$, contradicting the fact that $w_{\alpha}$ is a translation in $W_{J_{\alpha}}$ (as it has infinite order by Lemma~\ref{Lemma: straight element J essental}).
	
	Since $\nub(\widetilde{w})$ is closed, Proposition~\ref{prop: root groups in the nub} (whose hypotheses are satisfied by Remark~\ref{remark: schematic-completion is Kac--Moody completion}) now implies that $u \in \nub(\widetilde{w})$. As $g = u\inv (ug)$, it remains to show that $ug \in \nub(\widetilde{w})$. As $\Delta^{\imp} \cap \Delta^+_K=\Delta^{\imp}\cap \Delta^+_{K'}$ where $K':=\Jperp \cup \Jaff$, the conclusion then follows from Corollary~\ref{cor: imaginary root groups are contracted}  as it implies that 
		\[ ug \in U^{\ma}_{\Delta^{\imp}} \cap U^{\ma}_{\Delta^+_K} = U^{\ma}_{\Delta^{\imp} \cap \Delta^+_K} =U^{\ma}_{\Delta^{\imp} \cap \Delta^+_{K'}}\leq \nub(\widetilde{w}). \qedhere \]
\end{proof}

\subsection{Upper bound for the nub}\label{subsection:UBFTNKM}
We next show that the lower bound for the nub that we obtained in Corollary~\ref{corollary:main_lower_bound_nubKM} is also an upper bound, building on the results of \S\ref{subsection:upperbound_generalRGD}.

\begin{lemma}\label{Lemma: Coxeter group element with support J commutes with imaginary root with support in Jperp}
	Let $J \subseteq I$ and $w\in W_J$. Then $\widetilde{w}$ centralizes $U_{\alpha}$ for all $\alpha \in \Delta^{+}(\Jperp)$.
\end{lemma}
\begin{proof}
Let $\alpha \in \Delta^{+}(\Jperp)$ and $j\in J$. We have to show that $\widetilde{s}_j$ centralizes $U_{\alpha}$.  As $U^{\ma}_{\Delta^+(J \cup \Jperp)} = U^{\ma}_{\Delta^+(J)} \times U^{\ma}_{\Delta^+(\Jperp)}$ by \cite[Lemma~$8.58(4)$]{Ma18}, $U_{\alpha_j}$ centralizes $U_{\alpha}$. Hence $U_{-\alpha_j}$ also centralizes $U_{\alpha}$ since $\widetilde{s}_jU_{-\alpha_j}\widetilde{s}_j\inv=U_{\alpha_j}$ centralizes $\widetilde{s}_jU_{\alpha}\widetilde{s}_j\inv=U_{s_j\alpha}=U_{\alpha}$. As  $\widetilde{s}_j\in U_{\alpha_j}U_{-\alpha_j}U_{\alpha_j}$, the claim follows.
\end{proof}

\begin{lemma}\label{Lemma: Coxeter group element commutes with imaginary root group}
	Let $J\subseteq I$ and let $K\subseteq J$ be an affine component of $J$. Then there exists $N\in\NN^*$ such that $\widetilde{w}^N$ centralizes $U_{\alpha}$ for all $w\in W_J$ and $\alpha \in \Delta^{\imp}(K)$.
\end{lemma}
\begin{proof}
	Let $K':=K^{\perp}\cap J$ be the union of all components of $J$ different from $K$, so that $J=K\cup K'$. Let $\delta\in\Delta^{\imp}$ such that $\Delta^{\imp}(K)=\NN^*\delta$ (see \cite[Proposition~6.14(2)]{Ma18}). Set $M:=\max_{n\in\NN^*}\dim\mathfrak{g}_{n\delta}\in\NN$ (see \cite[Theorem~8.7]{Kac90}) and $N:=(|k|^M)!$. Let $w\in W_J=W_K\times W_{K'}$. Then $w\delta=\delta$, and hence $w^*\g_{n\delta}=\g_{n\delta}$ for all $n\in\NN$. In particular, $(w^*)^N$ is the identity on $\g_{n\delta}$ for all $n\in\NN^*$. On the other hand, if we set $\Psi:=\NN^*\delta$, an element $g\in U_{\alpha}$ with $\alpha\in\Delta^{\imp}(K)$ has standard form $g=\prod_{x\in \mc{B}_{\Psi}}[\exp](\lambda_xx)$ for some $\lambda_x\in k$, and hence
	\[ \widetilde{w}^Ng\widetilde{w}^{-N}=\prod_{x\in \mc{B}_{\Psi}}[\exp]((w^*)^N(\lambda_xx))=\prod_{x\in \mc{B}_{\Psi}}[\exp](\lambda_xx)=g. \qedhere \]
\end{proof}

\begin{lemma}\label{Lemma: Kac--Moody invariant}
	Let $w\in W$ and $K\subseteq I$ such that $w\inv W_Kw=W_K$. Let $g\in U^{\ma}_{\Delta^+_{K}}$ with $\widetilde{w}\inv g \widetilde{w} \in U^{\map}$. Then $\widetilde{w}\inv g \widetilde{w} \in U^{\ma}_{\Delta^+_{K}}$.
\end{lemma}
\begin{proof}
In view of Example~\ref{example:Deltawdecomposition}, we have $g\in U^{\ma}_{\Delta^+_K}\cap U^{\ma}_{\Delta^+\setminus\Delta_w}=U^{\ma}_{\Delta^+_K\cap \Delta^+\setminus\Delta_w}$. It thus remains to see that $w\inv(\Delta^+_K\cap \Delta^+\setminus\Delta_w)\subseteq\Delta^+_K$.
Since $w\inv W_Kw=W_K$, we have $w\inv \alpha_i\in\Delta(K)$ for all $i\in K$ (see Lemma~\ref{Lemma: support of a root irreducible}) and hence $w\inv \Delta(K)\subseteq\Delta(K)$. Thus if $\alpha\in \Delta^+_K\cap \Delta^+\setminus\Delta_w$, then $\alpha\in \Delta\setminus\Delta(K)$ and $w\inv\alpha\in\Delta^+$, so that $w\inv\alpha\in \Delta^+\setminus\Delta(K)=\Delta^+_K$, as desired.
\end{proof}

 \begin{proposition}\label{prop: nub contained in Delta_perp, affine, sphw}
	Let $w\in W$ be straight and standard and let $J = \supp(w)$. Set $K := \Jperp \cup \Jsph^w \cup \Jaff$. Then $\nub(\widetilde{w}) \leq U^{\ma}_{\Delta^+_K}$.
\end{proposition}
\begin{proof}
	By Proposition~\ref{prop: Upper bound for nub} and Lemma~\ref{Lemma: Unique decomposition} we have $\nub(\widetilde{w}) \leq U^{\map} = U^{\ma}_{\Delta^+(K)} \cdot U^{\ma}_{\Delta^+_K}$. Moreover, Lemma~\ref{Lemma: Fixator of orthogonal residues} and Theorem~\ref{Theorem: Upper bound for nub} imply that $\nub(\widetilde{w}) \leq \Fix(\Sigma \cap R_K(c))$. Note that $U^{\ma}_{\Delta^+_K} \leq \Fix(R_K(c))$ by (\ref{equation:UmaDelta+subK fixes R_K(c)}). Since $U^{\ma}_{\Delta^+(K)}\cap \Fix(\Sigma\cap R_K(c))=U^{\ma}_{\Delta^{\imp}(K)}$ by \cite[Exercise~$8.112(3)$]{Ma18} (or \cite[proof of Lemma~5.4]{Ma14}), we deduce that
	\begin{equation}\label{Equation: nub contained in Deltaim(K) cdot Delta_K}
		\nub(\widetilde{w}) \leq U^{\ma}_{\Delta^{\imp}(K)} \cdot U^{\ma}_{\Delta^+_K}.
	\end{equation}
	Note also that $\Delta^{\imp}(K) = \{ \alpha \in \Delta^{\imp} \mid \supp(\alpha) \subseteq \Jaff \cup \Jperp \}$ since $\Jsph^w$ is a union of finite type components of $K$. By Lemmas~\ref{Lemma: Coxeter group element with support J commutes with imaginary root with support in Jperp} and \ref{Lemma: Coxeter group element commutes with imaginary root group}, up to replacing $\widetilde{w}$ by some nontrivial power, we may assume that $\widetilde{w}$ centralizes each root group corresponding to a root in $\Delta^{\imp}(\Jaff) \cup \Delta^{\imp}(\Jperp) = \Delta^{\imp}(K)$.

Recall that we set $\Psi(n):=\{\alpha\in\Delta^+ \ | \ \height(\alpha)\geq n\}$.	For each $n\geq 1$, set $$\Psi_n :=\Psi(n)\cup \Delta^+_K= \{ \alpha \in \Delta^+(K) \mid \height(\alpha) \geq n \} \cup \Delta^+_K.$$ Note that $$U^{\ma}_{\Psi_n}=U^{\ma}_{\Delta^+(K)\cap\Psi(n)}\cdot U^{\ma}_{\Delta^+_K}=U^{\ma}_n U^{\ma}_{\Delta^+_K}$$ by Lemma~\ref{Lemma: Unique decomposition}, and hence $U^{\ma}_{\Psi_n}$ is open.
	
	\emph{Claim: $\nub(\widetilde{w}) \cap U^{\ma}_{\Psi_n}$ is $\widetilde{w}$-stable.}
	
	Indeed, let $g\in \nub(\widetilde{w}) \cap U^{\ma}_{\Psi_n}$ and $\epsilon\in\{\pm 1\}$. As $\nub(\widetilde{w})$ is $\widetilde{w}$-stable, we have to show that $\widetilde{w}^{\epsilon} g \widetilde{w}^{-\epsilon} \in U^{\ma}_{\Psi_n}$. By (\ref{Equation: nub contained in Deltaim(K) cdot Delta_K}) there exist $u\in U^{\ma}_{\Delta^{\imp}(K)}$ and $v\in U^{\ma}_{\Delta^+_K}$ with $g = uv$. Then $u = gv\inv \in U^{\ma}_{\Psi_n}$. Note that $\widetilde{w}^{\epsilon} g \widetilde{w}^{-\epsilon} \in \nub(\widetilde{w}) \leq U^{\map}$ as well as $\widetilde{w}^{\epsilon} u \widetilde{w}^{-\epsilon} = u$ (as assumed earlier in the proof), and hence $\widetilde{w}^{\epsilon} v \widetilde{w}^{-\epsilon} \in U^{\map}$. Since $w$ normalizes $W_{\Jperp}$, $W_{\Jsph^w}$ and $W_{\Jaff}$, it normalizes $W_K$. Lemma~\ref{Lemma: Kac--Moody invariant} then implies $\widetilde{w}^{\epsilon} v \widetilde{w}^{-\epsilon} \in U^{\ma}_{\Delta^+_{K}}$.	Hence
	\[ \widetilde{w}^{\epsilon} g \widetilde{w}^{-\epsilon} = \widetilde{w}^{\epsilon} u \widetilde{w}^{-\epsilon} \cdot \widetilde{w}^{\epsilon} v \widetilde{w}^{-\epsilon} \in u U^{\ma}_{\Delta^+_K} \subseteq U^{\ma}_{\Psi_n},  \]
	yielding the claim.
	
	We infer that $\nub(\widetilde{w}) \cap U^{\ma}_{\Psi_n}$ is a relatively open, $\widetilde{w}$-stable subgroup of $\nub(\widetilde{w})$. Lemma~\ref{Lemma: Facts about nub} now implies that $\nub(\widetilde{w}) \cap U^{\ma}_{\Psi_n} = \nub(\widetilde{w})$, that is, $\nub(\widetilde{w}) \leq U^{\ma}_{\Psi_n}$ for each $n\in\NN^*$. As $\bigcap_{n\in\NN^*}U^{\ma}_{\Psi_n}=\bigcap_{n\in\NN^*} U^{\ma}_nU^{\ma}_{\Delta^+_K}=\overline{U^{\ma}_{\Delta^+_K}}=U^{\ma}_{\Delta^+_K}$, this proves the proposition.
\end{proof}

\begin{corollary}\label{cor: Main result}
	Let $w\in W$ be straight and standard and let $J = \supp(w)$. Set $L_w := \{ w^zc \mid z\in \ZZ \}\subseteq \mc{B}$ and $K := \Jperp \cup \Jsph^w \cup \Jaff$. Then
	\[ \nub(\widetilde{w}) = \Fix(L_w) \cap U^{\ma}_{\Delta^+_K}. \]
\end{corollary}
\begin{proof}
The inclusion $\supseteq$ holds by Corollary~\ref{corollary:main_lower_bound_nubKM}. Conversely, recalling that $U^{\map}$ is admissible (see Remark~\ref{remark: schematic-completion is Kac--Moody completion}), the inclusion $\nub(\widetilde{w}) \subseteq \Fix(L_w)$ follows from Proposition~\ref{prop: Upper bound for nub}, while $\nub(\widetilde{w}) \subseteq U^{\ma}_{\Delta^+_K}$ holds by Proposition~\ref{prop: nub contained in Delta_perp, affine, sphw}.
\end{proof}

Recall that, for $w\in W$ straight and standard with support $J\subseteq S$, we set
$$K(w) = \Delta^+ \setminus \left( \Delta_{w+} \cup \Delta_{w-} \cup \Delta^+(J^{\perp} \cup J_{\mathrm{aff}} \cup \Jsph^w) \right)$$
as well as
$$\Delta_{w\pm} = \{ \alpha \in \Delta^+ \mid \exists n\in \NN: w^{\pm n} \alpha \in \Delta^- \},$$
so that $\Delta_{w+}\cup\Delta_{w-} \subseteq \Delta^{\rep}$ is the set of positive real roots which do not contain $\{ w^z \mid z\in \ZZ \}\subseteq\Sigma(W,S)$.

\begin{lemma}\label{lemma:FixLDeltawpm}
Let $w\in W$ be straight and standard, and let $J=\supp(w)$. Set $L_w := \{ w^zc \mid z\in \ZZ \}\subseteq \mc{B}$. Then
$$\Fix(L_w)=U^{\ma}_{\Delta^{+}\setminus (\Delta_{w+}\cup\Delta_{w-})}.$$
\end{lemma}
\begin{proof}
Since $U_{\alpha}\subseteq\Fix(L_w)$ whenever $\alpha\in\Delta^{+}\setminus (\Delta_{w+}\cup\Delta_{w-})$, the inclusion $\supseteq$ is clear. For the reverse inclusion, note that $\Delta_{w+}\cap\Delta_{w-}=\varnothing$ by \cite[Lemma~3.2]{Ma14}, so that we have a unique decomposition $U^{\map}=U^{\ma}_{\Delta_{w+}}\cdot U^{\ma}_{\Delta_{w-}}\cdot U^{\ma}_{\Delta^{+}\setminus (\Delta_{w+}\cup\Delta_{w-})}$ (see \cite[Lemma~8.54]{Ma18}). We thus have to show that $\Fix(L_w)\cap U^{\ma}_{\Delta_{w+}}\cdot U^{\ma}_{\Delta_{w-}}=\{1\}$. Let $u_{\pm}\in U^{\ma}_{\Delta_{w\pm}}$, and suppose that $u_+u_-\in\Fix(L_w)$. 
Set $L_{\pm}:=\{ w^{\pm n}c \mid n\in \NN \}$. Note that an element $u\in U^{\map}$ belongs to $\Fix(L_{\pm})$ if and only if $u\in U^{\ma}_{\Delta^+\setminus\Delta_{w\mp}}$. Since $u_{\pm}\in U^{\ma}_{\Delta^+\setminus\Delta_{w\mp}}$, the fact that $u_+u_-\in\Fix(L_w)$ implies that $u_{\pm}$ fixes $L_{\mp}$, and hence $u_{\pm}\in U^{\ma}_{\Delta_{w\pm}}\cap U^{\ma}_{\Delta^+\setminus\Delta_{w\pm}}=\{1\}$, as desired.
\end{proof}

\begin{corollary}\label{corollary:ThmA}
Theorem~\ref{Thmintro: Main result} holds.
\end{corollary}
\begin{proof}
This readily follows from Corollary~\ref{cor: Main result} and Lemma ~\ref{lemma:FixLDeltawpm}.
\end{proof}

\subsection{Closure of contraction groups}\label{Subsection: Closure of contraction groups}

\begin{proposition}\label{prop:trichotomy_contraction_root_groups}
Let $w\in W$ be straight and standard, and let $J=\supp(w)$. Let $K_1(w)$ be the union of $\Delta^+(J^{\perp} \cup \Jsph^w)$ and $\Delta^+(J_{\mathrm{aff}})\setminus (\Delta_{w+}\cup\Delta_{w-})$, and set $K_2(w):=\Delta_{w+}$. Then there exists $N\in\NN^*$ such that for any $\alpha\in\Delta^+$, exactly one of the following holds:
\begin{enumerate}
\item
$w^N\alpha=\alpha$. This case occurs if and only if $\alpha\in K_1(w)$.
\item
$\height(w^n\alpha)\to -\infty$ as $n\to +\infty$. This case occurs if and only if $\alpha\in K_2(w)$.
\item
$\height(w^n\alpha)\to +\infty$ as $n\to +\infty$. This case occurs if and only if $\alpha\in K_3(w):=\Delta^+\setminus (K_1(w)\cup K_2(w))$.
\end{enumerate}
In particular, the sets $K_1(w)$, $K_2(w)$ and $K_3(w)$ are closed and form a partition of $\Delta^+$. Moreover, $K(w)= K_3(w)\cap K_3(w\inv)$ is also closed. Finally, $\Delta_{w-}\subseteq K_3(w)$. 
\end{proposition}
\begin{proof}
Let $\alpha\in\Delta^+$. 

Assume first that $w^n\alpha=\alpha$ for some $n\in\NN^*$. Then $\supp(\alpha)\subseteq J^{\perp}\cup \Jsph^w\cup \JaffCox$ if $\alpha\in\Phi_+$ by Lemma~\ref{Lemma: root groups in the nub} and $\supp(\alpha)\subseteq J^{\perp}\cup J_{\mathrm{aff}}$ if $\alpha\in\Delta^{\imp}$ by Theorem~\ref{thm: root groups are contracted}. If $\supp(\alpha)\subseteq J^{\perp}$, then $w\alpha=\alpha$. If $\supp(\alpha)\subseteq \Jsph^w$, there is some $N_1\in\NN^*$ (depending only on $(W,S)$) such that $w^{N_1}\alpha=\alpha$. If $\alpha\in\Delta^{\imp}$ and $\supp(\alpha)\subseteq J_{\mathrm{aff}}$ then $\alpha\notin\Delta_{w+}\cup\Delta_{w-}$ and $w\alpha=\alpha$. And if $\supp(\alpha)\subseteq \JaffCox$ and $\alpha\in\Phi_+$, then $\alpha\in \Delta_+(J_{\mathrm{aff}})\setminus(\Delta_{w+}\cup\Delta_{w-})$ and there is some $N_2\in\NN^*$ (depending only on $(W,S)$) such that $w^{N_2}\alpha=\alpha$ by Lemma~\ref{Lemma: affine translations}(2). Thus, if $w^n\alpha=\alpha$ for some $n\in\NN^*$, then $w^N\alpha=\alpha$ where $N:= N_1 N_2$ depends only on $(W,S)$, and this occurs exactly when $\alpha\in K_1(w)$.

Assume next that $w^n\alpha\neq\alpha$ for all $n\in\NN^*$ (that is, we are not in case (1)). Then $|\height(w^n\alpha)|\to +\infty$ as $n\to +\infty$. Recall from \cite[Lemma~3.2]{Ma14} that the function $\ZZ\to\{\pm 1\}:n\mapsto \mathrm{sign}(w^n\alpha)$ is monotonic, where $\mathrm{sign}(\beta):=\pm 1$ if $\beta\in\Delta^{\pm}$. In particular, either $\height(w^n \alpha) \to -\infty$ or $\height(w^n \alpha) \to +\infty$ as $n \to +\infty$, that is, we are either in case (2) or (3).
If $\height(w^n\alpha)\to -\infty$ as $n\to +\infty$, then $\alpha\in K_2(w)$, and conversely if $\alpha\in K_2(w)$, then $w^n\alpha\in\Phi_-$ for all large enough $n\in\NN^*$ and hence $\height(w^n\alpha)\to -\infty$ as $n\to +\infty$. Finally,  $\height(w^n\alpha)\to +\infty$ as $n\to +\infty$ if and only if we are not in cases (1) or (2) if and only if $\alpha\notin K_1(w)\cup K_2(w)$.
This proves the claimed trichotomy (1)--(3). The second assertion is now clear.

To see that $K(w)= K_3(w)\cap K_3(w\inv)$, note that $J_{\sph}^w=J_{\sph}^{w\inv}$, while $\Delta_{w\inv +}=\Delta_{w-}$ and $\Delta_{w\inv -}=\Delta_{w+}$. Hence 
$$K_1(w)\cup K_2(w)\cup K_1(w\inv)\cup K_2(w\inv)=\Delta_{w+} \cup \Delta_{w-} \cup \Delta^+(J^{\perp} \cup J_{\mathrm{aff}} \cup \Jsph^w),$$
yielding the claim.

Finally, $\Delta_{w-}\subseteq K_3(w)$, because if $\alpha\in\Delta_{w-}$ then $w^n\alpha\in\Delta^+$ for all $n\in\NN$ by \cite[Lemma~3.2]{Ma14} and hence $\height(w^n\alpha)\to +\infty$ as $n\to +\infty$.
\end{proof}

For $w\in W$, set 
$$U^-_{w+}:=\langle U_{\alpha} \ | \ \alpha\in -\Delta_{w+}\rangle.$$

\begin{theorem}\label{thm:overlinecon}
Let $w\in W$ be straight and standard, and let $J=\supp(w)$. Then
$$\overline{\con(\widetilde{w})}=U^-_{w+}\cdot U^{\ma}_{K_3(w)}.$$
\end{theorem}
\begin{proof}
Note first that if $\alpha\in -\Delta_{w+}$, then $w^n\alpha\in \Delta_{w-}$ for some $n\in\NN^*$ and hence $U_{\alpha}\subseteq \con(\widetilde{w})$ by Proposition~\ref{prop:trichotomy_contraction_root_groups}, so that $\con(\widetilde{w})\supseteq U^-_{w+}$. Since $U_{\alpha}\subseteq \con(\widetilde{w})$ for all $\alpha\in K_3(w)$ by Proposition~\ref{prop:trichotomy_contraction_root_groups}, the inclusion $\overline{\con(\widetilde{w})}\supseteq U^-_{w+}\cdot U^{\ma}_{K_3(w)}$ follows.

For the reverse inclusion, recall that $\overline{\con(\widetilde{w})}=\con(\widetilde{w})\cdot \nub(\widetilde{w})$ (see \cite[Corollary~3.30]{BW04}). Since $\nub(\widetilde{w})=U^{\ma}_{K(w)}\subseteq U^{\ma}_{K_3(w)}$ by Theorem~\ref{Thmintro: Main result}, it remains to check that $\con(\widetilde{w})\subseteq U^-_{w+}\cdot U^{\ma}_{K_3(w)}$.

Note that if $g\in G^{\pma}$ belongs to $\con(\widetilde{w})$, then there is some $n\in\NN$ such that $\widetilde{w}^ng\widetilde{w}^{-n}\in U^{\map}$. Recall from Example~\ref{example:Deltawdecomposition} we have a unique decomposition $U^{\map}=U^{\ma}_{\Delta_{w^n}}\cdot U^{\ma}_{\Delta^+\setminus\Delta_{w^n}}$, where $\Delta_{w^n}=\{\alpha\in\Delta^+ \ | \ w^{-n}\alpha\in\Delta^-\}$. Hence $g\in\widetilde{w}^{-n} U^{\ma}_{\Delta_{w^n}}\widetilde{w}^n\cdot U^{\map}\subseteq U^-_{w+}\cdot U^{\map}$. Thus, $\con(\widetilde{w})\subseteq U^-_{w+}\cdot U^{\map}$, and as $U^-_{w+}\subseteq \con(\widetilde{w})$, it is sufficient to show that $\con(\widetilde{w})\cap U^{\map}\subseteq U^{\ma}_{K_3(w)}$.

Let $g\in \con(\widetilde{w})\cap U^{\map}$. By \cite[Lemma~8.54]{Ma18} and Proposition~\ref{prop:trichotomy_contraction_root_groups}, we have a unique decomposition $U^{\map}=U^{\ma}_{K_1(w)}\cdot U^{\ma}_{K_2(w)}\cdot U^{\ma}_{K_3(w)}$, and we write $g=u_1u_2u_3$ accordingly, with $u_i\in U^{\ma}_{K_i(w)}$. Since $\widetilde{w}^ng\widetilde{w}^{-n}$ converges to $1$ as $n\to +\infty$, it readily follows from the trichotomy in Proposition~\ref{prop:trichotomy_contraction_root_groups} that $u_1=u_2=1$, yielding the claim.
\end{proof}

\subsection{Proof of Corollaries~\ref{corintro:nuboverline} and \ref{corintro:congbar}, concluded}\label{subsection:proofCorBC}

To prove Corollary~\ref{corintro:nuboverline}, note that we may assume that $g=\widetilde{w}$ for some straight and standard $w\in W$ by Proposition~\ref{prop: The nub only for straight and standard elements}. By Theorem~\ref{Thmintro: Main result}, it is thus sufficient to show that $U_{\alpha}\subseteq \con(\widetilde{w})\cap\con(\widetilde{w}\inv)$ for all $\alpha\in K(w)$, which follows from Proposition~\ref{prop:trichotomy_contraction_root_groups}.

Corollary~\ref{corintro:congbar} follows from Theorem~\ref{thm:overlinecon},
since $K_3(w)=K(w)\cup\Delta_{w-}$.

\subsection{Proof of Corollary~\ref{Corintro: trivial nub}}

Given a subset $J\subseteq I$, we let $J^{\infty}$ denote the \emph{essential part} of $J$, that is, the union of the components of $J$ of non-finite type.

\begin{proposition}
Let $w\in W$ be straight, and let $J = \supp(w)$. Then $\nub(\widetilde{w}) = \{1\}$ if and only if $J$ is a union of affine components of $I$.
\end{proposition}
\begin{proof}
	By Lemma~\ref{Lemma: infinite order and standard elements} we may assume without loss of generality that $w$ is straight and standard. Assume first that $\nub(\widetilde{w}) = \{1\}$. Then Theorem~\ref{Thmintro: Main result} yields $K(w)=\varnothing$ and hence $\Delta^{\mathrm{im+}}=\Delta^{\mathrm{im+}}\setminus K(w)=\Delta^{\mathrm{im+}}(J^{\perp}\cup J_{\mathrm{aff}})$. This implies that $I^{\infty}\subseteq J^{\perp}\cup J_{\mathrm{aff}}$ (see e.g.\ \cite[Theorem~5.6]{Kac90}) and hence that $J=J_{\mathrm{aff}}$ and $I=J^{\perp}\cup J_{\mathrm{aff}}$, as desired. Conversely, if $I=J^{\perp}\cup J_{\mathrm{aff}}$, then $K(w)=\varnothing$ and hence $\nub(\widetilde{w}) = \{1\}$ by Theorem~\ref{Thmintro: Main result}.
\end{proof}

\subsection{Proof of Corollary~\ref{Corintro: Same nub}}

\subsubsection{Preliminaries on the Davis complex}

Basics on the Davis complex and on $\mathrm{CAT}(0)$-spaces can be found in \cite{Davis08}, \cite{AB08} and \cite{Bridson_Haelfiger:Metric_spaces_of_non-positively_curvature}. 

Let $X$ be the Davis complex of $(W, S)$. Thus $X$ is a $\mathrm{CAT}(0)$ metric realization of $\Sigma := \Sigma(W, S)$, whose $\mathrm{CAT}(0)$ metric we denote by $\dcat$, on which $W$ acts by cellular isometries. For instance, if $(W, S)$ is of affine type, then $X$ is the standard geometric realization of $\Sigma$ equipped with the Euclidean distance. We identify the walls, roots and simplices of $\Sigma$ with their (closed) realization in $X$. In particular, the wall $\partial \alpha$ is the intersection of the opposite roots $\alpha$ and $-\alpha$ in $X$.

Consider the $W$-action on $X$. For an element $w\in W$ we let
\[ \vert w \vert := \inf\{ \dcat(x, wx) \mid x\in X \} \in [0, \infty) \]
denote its \emph{translation length} and we set
\[ \mathrm{Min}(w) := \{ x\in X \mid \dcat(x, wx) = \vert w \vert \}. \]
A classical result of M.~Bridson (see \cite{Bridson:On_the_semisimplicity_of_polyhedral_isometries}) asserts that for such an action $\mathrm{Min}(w) \neq \emptyset$ for any $w\in W$. More precisely, if $w$ has infinite order, then $\vert w \vert > 0$ and $\mathrm{Min}(w)$ is the union of all $w$-axes, where a \emph{$w$-axis} is a geodesic line stabilized by $w$ (on which $w$ then acts by translation).\\ Note that any nonempty subset $C \subseteq X$ which is closed, convex and $w$-invariant contains a $w$-axis by \cite[Proposition~6.2]{Bridson_Haelfiger:Metric_spaces_of_non-positively_curvature}.

Given a root $\alpha$, the wall $\partial \alpha$ as well as the closed and open half-spaces $\alpha$ and $\alpha \setminus \partial \alpha$ are convex subsets of $X$. Moreover, for any geodesic line $l$ and any wall $m$ such that $l \cap m \neq \emptyset$ either $l \subseteq m$ or $\vert l \cap m \vert = 1$ (see \cite[Lemma~3.4]{NV02}). In the latter case $l$ is called \emph{transverse} to $m$. \\
Let $w\in W$ be of infinite order. A wall $m$ is \emph{$w$-essential} if it intersects some (equivalently, any -- see \cite[Lemma~2.5]{Caprace-Marquis13}) $w$-axis in a single point.

\subsubsection{Elements with the same nub}

For $w\in W$ we set as before
\[ L_w := \{ w^z c \mid z\in \ZZ \}. \]

\begin{lemma}\label{lemma:Lwandessential}
Let $w\in W$ be straight, and let $\alpha\in\Phi_+$. Then the following are equivalent:
\begin{enumerate}
\item
$\alpha\in\Delta_{w+}\cup\Delta_{w-}$.
\item
$\alpha\not\supseteq L_w$.
\item
$\alpha\not\supseteq \conv(L_w)$.
\item
$\partial\alpha$ is a $w$-essential wall.
\end{enumerate} 
Moreover, if these conditions hold, then $r_{\alpha}\in\Pc(w)$ and hence $\supp(\alpha) \subseteq \supp(w)$.
\end{lemma}
\begin{proof}
The equivalence (1)$\Leftrightarrow$(2) is clear and (2)$\Leftrightarrow$(3) follows from \cite[Lemma~3.44]{AB08}. To see that (2,3)$\Leftrightarrow$(4), note that since $\conv(L_w)$ is a closed convex $w$-invariant subspace of the Davis realization of $\Sigma(W,S)$, it contains a $w$-axis $l_w$. If $\alpha\supseteq \conv(L_w)$, then $\alpha\supseteq l_w$ and hence $\partial\alpha$ is not $w$-essential. Conversely, suppose that $\alpha\not\supseteq L_w$. Assume for a contradiction that $\partial\alpha$ is not $w$-essential (in particular, if $l_w\cap\partial\alpha\neq\varnothing$ then $l_w\subseteq\partial\alpha$). Then $l_w\subseteq\epsilon\alpha$ for some $\epsilon\in\{\pm\}$. Since $w$ is straight, there is a unique $z\in\ZZ$ such that $w^zc$ and $w^{z+1}c$ lie on different sides of $\partial\alpha$, that is, there are some $N\in\ZZ$ and $\mu\in\{\pm 1\}$ such that $L'_w:=\{w^{N+\mu n}c \ | \ n\in\NN\}\subseteq -\epsilon\alpha$ while $L_w\setminus L'_w\subseteq\epsilon\alpha$. Since $L_w$ and $l_w$ are at bounded Hausdorff distance (as $L_w$ and $l_w$ are both the $\langle w\rangle$-orbit of a bounded set), there is some $d\in\NN$ such that every chamber of $L'_w$ is at distance at most $d$ from $\partial\alpha$. Let $r\in\NN$ be such that there are at most $r$ walls intersecting the ball of radius $d$ around any chamber. Since the chamber $w^{N+\mu r}c\in L'_w$ is at distance at most $d$ from the $r+1$ walls $\partial\alpha,w^{\mu}\partial\alpha,\dots,w^{\mu r}\partial\alpha$, we find some $s>0$ such that $w^s\partial\alpha=\partial\alpha$ (and hence $w^{2s}\alpha=\alpha$). But then $\alpha$ contains $w^{2s\ZZ}c$ and hence $L_w$, yielding the desired contradiction.

The last statement follows from \cite[Lemma~2.7]{Caprace-Marquis13} and Lemma~\ref{Lemma: straight element J essental}.
\end{proof}

For a straight element $w\in W$ with $J=\supp(w)$, and a union $J'$ of components of $J$, we write $w_{J'}$ for the element of $W_{J'}$ such that $w=w_{J'}w'$ for some $w'\in W_{J\setminus J'}$.

\begin{lemma}\label{lemma:conhullcomponents}
Let $w\in W$ be straight and let $J=\supp(w)$. Write $J_1,\dots,J_n$ for the components of $J$, and set $w_i:=w_{J_i}$. Then the following hold.
\begin{enumerate}
\item
$\Delta_{w_i+}\cup\Delta_{w_i-}=\{\alpha \in \Delta_{w+}\cup\Delta_{w-} \ | \ \supp(\alpha)\subseteq J_i\}$ for all $i=1,\dots,n$. 
\item
$\Delta_{w+}\cup\Delta_{w-}=\bigcup_{i=1}^n(\Delta_{w_i+}\cup\Delta_{w_i-})$.
\end{enumerate}
\end{lemma}
\begin{proof}
(1) Note that if $\alpha\in\Delta_{w_i\pm}$, then $\supp(\alpha)\subseteq J_i$ by Lemma~\ref{lemma:Lwandessential}. Assuming now that $\supp(\alpha)\subseteq J_i$, we have $w^n\alpha=w_i^n\alpha$ for all $n\in\ZZ$, so that (1) is clear.

(2) Since $\supp(\alpha)\subseteq J$ for any $\alpha\in\Delta_{w\pm}$ by Lemma~\ref{lemma:Lwandessential}, (2) follows from (1). 
\end{proof}

\begin{lemma}\label{lemma:convconv}
Let $w,v\in W$ be straight, and write $J=\supp(w)$ and $K=\supp(v)$. Then the following are equivalent:
\begin{enumerate}
\item
$\conv(L_w)=\conv(L_v)$.
\item
$J=K$ and $\conv(L_{w_{J'}})=\conv(L_{v_{J'}})$ for each component $J'$ of $J$.
\end{enumerate}
\end{lemma}
\begin{proof}
Write $J_1,\dots,J_n$ for the components of $J$. Note that by Lemma~\ref{lemma:Lwandessential}, for any two straight elements $x,y\in W$, we have $\conv(L_x)=\conv(L_y)$ if and only if $\Delta_{x+}\cup\Delta_{x-}=\Delta_{y+}\cup\Delta_{y-}$ (recall that the convex hull of a subset $Z$ is the intersection of all the roots containing $Z$).

If (1) holds, then $w \in \conv(L_w) = \conv(L_v) \subseteq W_K$, which implies $J = \supp(w) \subseteq K$, and similarly $K \subseteq J$ and hence $J = K$. Moreover, $\Delta_{w+}\cup\Delta_{w-}=\Delta_{v+}\cup\Delta_{v-}$. Thus, for each $i\in\{1,\dots,n\}$, the sets $\Delta_{w_{J_i}+}\cup\Delta_{w_{J_i}-}$ and $\Delta_{v_{J_i}+}\cup\Delta_{v_{J_i}-}$ coincide by Lemma~\ref{lemma:conhullcomponents}(1), and hence $\conv(L_{w_{J_i}})=\conv(L_{v_{J_i}})$, yielding (2).

Conversely, if (2) holds, then $\Delta_{w_{J_i}+}\cup\Delta_{w_{J_i}-}=\Delta_{v_{J_i}+}\cup\Delta_{v_{J_i}-}$ for each $i\in\{1,\dots,n\}$. Thus $\Delta_{w+}\cup\Delta_{w-}=\Delta_{v+}\cup\Delta_{v-}$ by Lemma~\ref{lemma:conhullcomponents}(2), which implies (1).
\end{proof}

Recall that for $w$ straight and standard such that $\Pc(w)$ is of irreducible indefinite type, $\Delta(J^w_{\mathrm{sph}})$ is the set of roots $\alpha\in\Phi$ with $r_{\alpha}\in\Pc(w)$ that are stabilized by some nontrivial power of $w$, or equivalently such that $r_{\alpha}\in P_w^{\max}$ (see \cite[$\S$9.1]{Ma23}).

\begin{lemma}\label{lemma:descriptionPhiwsph}
Let $w\in W$ be straight and standard, and let $J=\supp(w)$. Assume that $W_J$ is of irreducible indefinite type. Then $$\Delta(J^w_{\mathrm{sph}})=\{ \alpha\in\Phi \ | \ \textrm{$\supp(\alpha)\subseteq J$ and $\partial\alpha$ intersects all $w$-essential walls}\}.$$
In particular, if $v\in W$ is straight and standard and such that $\conv(L_w)=\conv(L_v)$, then $J=\supp(v)$ and $\Delta(J^v_{\mathrm{sph}})=\Delta(J^w_{\mathrm{sph}})$.
\end{lemma}
\begin{proof}
If $\alpha\in \Delta(J^w_{\mathrm{sph}})$, then $\supp(\alpha)\subseteq J$ and $\partial\alpha$ is a closed convex $w^n$-invariant subspace of the Davis realization of $\Sigma(W,S)$ for some nonzero $n\in\NN$. Thus $\partial\alpha$ contains a $w$-axis, and hence intersects each $w$-essential wall. Conversely, suppose that $r_{\alpha}\in W_J$ and that $\partial\alpha$ intersects all $w$-essential walls. Assume for a contradiction that $w^n\partial\alpha\neq\partial\alpha$ for all $n>0$. Let $W_0$ denote as before a torsion-free finite index normal subgroup of $W$ and let $N\in\NN^*$ be such that $w^N\in W_0$. Pick a $w$-essential wall $m$. Then the $w$-essential walls $m_n=w^{Nn}m$ for $n\in\ZZ$ are pairwise disjoint and such that $r_{m_n}\in W_J$ (see \cite[Lemmas~2.6 and 2.7]{Caprace-Marquis13}). Similarly, the walls $m'_n:=w^{Nn}\partial\alpha$ for $n\in\ZZ$ are pairwise disjoint and such that $r_{m'_n}=w^{Nn}r_{\alpha}w^{-Nn}\in W_J$. Moreover, $m_n$ intersects $m'_{n'}$ for all $n,n'\in\ZZ$. This yields the desired contradiction by \cite[Lemma~9.4]{Ma23}.

For the second statement, note that if $K:=\supp(v)$, then $K=J$ by Lemma~\ref{lemma:convconv} and the sets of $w$-essential and $v$-essential walls coincide by Lemma~\ref{lemma:Lwandessential}.
\end{proof}

\begin{lemma}\label{lemma:KwJaffabsorb}
Let $w\in W$ be straight and let $J=\supp(w)$. Then
$$\Delta^+\setminus K(w)=\{\alpha\in\Delta^{\rep}(J\setminus\Jaff) \ | \ \alpha\not\supseteq \conv(L_{w_{J\setminus\Jaff}})\}\cup \Delta^+( \Jperp \cup \Jaff)\cup \Delta^+(J^w_{\mathrm{sph}}).$$
\end{lemma}
\begin{proof}
By definition, $\Delta^+\setminus K(w)=\Delta_{w+} \cup \Delta_{w-} \cup \Delta^+( \Jperp \cup \Jaff)\cup\Delta^+(J^w_{\mathrm{sph}})$. By Lemmas~\ref{lemma:Lwandessential} and \ref{lemma:conhullcomponents}, $$\Delta_{w+} \cup \Delta_{w-}= R\cup  \{\alpha\in\Delta^{\rep}(J\setminus\Jaff) \ | \ \alpha\not\supseteq \conv(L_{w_{J\setminus\Jaff}})\}$$
where $R:=\{\alpha \in \Delta_{w_{\Jaff}\pm} \ | \ \supp(\alpha)\subseteq \Jaff\}\subseteq\Delta^{+}(\Jaff)$. The lemma follows.
\end{proof}

\begin{proposition}
Let $w,v\in W$ be straight and standard, and let $J = \supp(w)$ and $K = \supp(v)$. Then the following are equivalent:
\begin{enumerate}
\item 
$\nub(\widetilde{w}) = \nub(\widetilde{v})$;
\item
$\Jperp \cup \Jaff = K^{\perp} \cup K_{\mathrm{aff}}$ and $\conv(L_{w_{J\setminus\Jaff}}) = \conv(L_{v_{K\setminus K_{\mathrm{aff}}}})$.
\item
$\Jperp \cup \Jaff = K^{\perp} \cup K_{\mathrm{aff}}$ and $J\setminus\Jaff=K\setminus K_{\mathrm{aff}}$ and $\conv(L_{w_{J'}}) = \conv(L_{v_{J'}})$ for every component $J'$ of $J\setminus\Jaff$ with $|J'|\geq 3$.
\end{enumerate}
\end{proposition}
\begin{proof}
(2)$\Leftrightarrow$(3): Note that $\conv(L_{w_{J\setminus\Jaff}}) = \conv(L_{v_{K\setminus K_{\mathrm{aff}}}})$ if and only if $J\setminus\Jaff=K\setminus K_{\mathrm{aff}}$ and $\conv(L_{w_{J'}}) = \conv(L_{v_{J'}})$ for every component $J'$ of $J\setminus\Jaff$ by Lemma~\ref{lemma:convconv}. Moreover, the condition $\conv(L_{w_{J'}}) = \conv(L_{v_{J'}})$ automatically holds if $|J'|<3$, since in that case $|J'|=2$ and $\conv(L_{w_{J'}}) = \Sigma(W_{J'},J')= \conv(L_{v_{J'}})$. This shows that (2)$\Leftrightarrow$(3).

(2,3)$\Rightarrow$(1): Recall that $\Jsph^w = \Jsph^{w_{J \setminus \Jaff}}$. By Lemma~\ref{lemma:descriptionPhiwsph}, we then have $\Delta(\Jsph^w)=\Delta(\Jsph^v)$. Lemma~\ref{lemma:KwJaffabsorb} now implies that $K(w)=K(v)$, so that (1) follows from Theorem~\ref{Thmintro: Main result}.

(1)$\Rightarrow$(2): By Theorem~\ref{Thmintro: Main result}, we have $K(w)=K(v)$. In particular, $$\Delta^{\imp}(J^{\perp}\cup J_{\textrm{aff}})=\Delta^{\imp}\setminus K(w)=\Delta^{\imp}\setminus K(v)=\Delta^{\imp}(K^{\perp}\cup K_{\textrm{aff}}).$$
Note for a subset $L\subseteq I$, we can recover the union $L^{\infty}$ of its irreducible components of non-finite type as the set of $i\in I$ belonging to the support of a root in $\Delta^{\imp}(L)$ (see \cite[Theorem~5.6]{Kac90}).
Hence 
\begin{equation}\label{eqn:JperpJaffinfty}
(J^{\perp}\cup J_{\textrm{aff}})^{\infty}=(K^{\perp}\cup K_{\textrm{aff}})^{\infty}.
\end{equation}
Lemma~\ref{lemma:KwJaffabsorb} then implies that $$X_J:=\{\alpha\in\Delta^{\rep}(J\setminus\Jaff) \ | \ \alpha\not\supseteq \conv(L_{w_{J\setminus\Jaff}})\}$$ and $X_K:=\{\alpha\in\Delta^{\rep}(K\setminus K_{\textrm{aff}}) \ | \ \alpha\not\supseteq \conv(L_{v_{K\setminus K_{\textrm{aff}}}})\}$ (which are respectively disjoint from $\Delta^+(J^{\perp}\cup \Jaff)$ and $\Delta^+(K^{\perp}\cup K_{\mathrm{aff}})$) only differ by finitely many roots. But this implies that $$X_J=X_K.$$ Indeed, if $X_J\neq X_K$, say there exists $\alpha\in X_J\setminus X_K$, then $\partial\alpha$ is $w_{J\setminus\Jaff}$-essential by Lemma~\ref{lemma:Lwandessential} and $\alpha\supseteq \conv(L_{v_{K\setminus K_{\textrm{aff}}}})$ --- recall from Lemma~\ref{lemma:Lwandessential} that 
$$X_K=\{\alpha\in\Delta^{\rep} \ | \ \alpha\not\supseteq \conv(L_{v_{K\setminus K_{\textrm{aff}}}})\}.$$
But then choosing $N\in\ZZ^*$ such that $\beta_n:=w_{J\setminus\Jaff}^{nN}\alpha\supseteq\alpha$ for all $n\in\NN$ (see \cite[Lemma~2.6]{Caprace-Marquis13}), the pairwise distinct roots $\beta_n$ ($n\in\NN$) belong to $X_J\setminus X_K$ by Lemma~\ref{lemma:Lwandessential}, a contradiction. 

Thus $X_J=X_K$, and hence $\Delta^{\rep} \setminus X_J = \Delta^{\rep} \setminus X_K$, that is, $$\{\alpha\in\Delta^{\rep} \ | \ \alpha\supseteq \conv(L_{w_{J\setminus\Jaff}})\} = \{\alpha\in\Delta^{\rep} \ | \ \alpha\supseteq \conv(L_{v_{K\setminus K_{\textrm{aff}}}})\},$$ so that $\conv(L_{w_{J\setminus\Jaff}})=\conv(L_{v_{K\setminus K_{\textrm{aff}}}})$. In particular, $J\setminus\Jaff=K\setminus K_{\textrm{aff}}$ by Lemma~\ref{lemma:convconv}. Therefore, together with (\ref{eqn:JperpJaffinfty}),
$$(J^{\perp}\cup J)^{\infty}=(J^{\perp}\cup J_{\textrm{aff}})^{\infty}\cup (J\setminus\Jaff)^{\infty}=(K^{\perp}\cup K_{\textrm{aff}})^{\infty}\cup (K\setminus K_{\textrm{aff}})^{\infty}=(K^{\perp}\cup K)^{\infty}.$$ 
As $J$ is essential, we have $J = J^{\infty}$. This implies that $( J \cup \Jperp ) \setminus ( \Jperp \cup J )^{\infty} = ( ( \Jperp \cup J )^{\infty} )^{\perp}$: indeed,  $J\cup J^{\perp}$ is the disjoint union of the commuting sets $$( \Jperp \cup J )^{\infty}=(\Jperp)^{\infty}\cup J^{\infty}=(\Jperp)^{\infty}\cup J\quad\textrm{and}\quad ( J \cup \Jperp ) \setminus ( \Jperp \cup J )^{\infty}=\Jperp \setminus ( \Jperp )^{\infty},$$
so that the inclusion $\subseteq$ is clear, and moreover $((\Jperp)^{\infty}\cup J)^{\perp}\subseteq\Jperp\setminus ( \Jperp )^{\infty}$, yielding the reverse inclusion as well. Since the same similarly holds for $K$, we conclude that $J\cup J^\perp=K\cup K^\perp$, and hence that 
$$\Jperp \cup \Jaff = \left( \Jperp \cup J \right) \setminus \left( J \setminus \Jaff \right) = \left( K^{\perp} \cup K \right) \setminus \left( K \setminus K_{\mathrm{aff}} \right) = K^{\perp} \cup K_{\mathrm{aff}},$$
yielding (2).
\end{proof}

\subsection{Proof of Corollary~\ref{Corintro: infinite depth}}

For a compact tdlc group $G$ and $\alpha \in \Aut(G)$ recall that the pair $(G, \alpha)$ has \emph{finite depth}, if there exists an open subgroup $V \leq G$ such that $\bigcap_{z\in \ZZ} \alpha^z(V) = \{1\}$. Note that $\nub(g)$ is a compact tdlc group which is $g$-stable, hence $\gamma_g \vert_{\nub(g)} \in \Aut(\nub(g))$. Here is a restatement of Corollary~\ref{Corintro: infinite depth}.

\begin{proposition}
	Let $w\in W$ be straight and such that $\nub(\widetilde{w}) \neq \{1\}$. Then $(\nub(\widetilde{w}), \gamma_{\widetilde{w}} \vert_{\nub(\widetilde{w})})$ does not have finite depth.
\end{proposition}
\begin{proof}
	Note that by Lemma~\ref{Lemma: infinite order and standard elements} it suffices to show the claim for $w$ straight and standard; we set $J:=\supp(w)$. Moreover, it is enough to show that $\bigcap_{z\in \ZZ} \widetilde{w}V_n\widetilde{w}^{-z} \neq \{1\}$ for each $n\in\NN$, where $V_n:=\nub(\widetilde{w})\cap U^{\ma}_n$. Fix $n\in\NN$. Since $\nub(\widetilde{w}) \neq \{1\}$, there exists by Corollary~\ref{Corintro: trivial nub} a component $J'$ of $J$ that is contained in a component $I'$ of $I$ of indefinite type. By \cite[Theorem~5.6]{Kac90}, there exists a root $\beta\in\Delta^{\imp}(I')$ with support $I'$ that is of minimal height in $W\beta$. Up to replacing $\beta$ by some multiple, we may further assume that $\height(\beta)\geq n$ (see \cite[Lemma~5.3]{Kac90}). Note that $U_{\beta}\subseteq\nub(\widetilde{w})$ by Theorem~\ref{Thmintro: Main result}. As $\height(w^z\beta)\geq\height(\beta)\geq n$ for all $z\in\ZZ$, we conclude that $\widetilde{w}^z U_{\beta}\widetilde{w}^{-z}=U_{w^z\beta}\in V_n$ for all $z\in\ZZ$. Hence $U_{\beta}\subseteq \bigcap_{z\in \ZZ} \widetilde{w}^z V_n\widetilde{w}^{-z}$, as desired.
\end{proof}

\bibliographystyle{amsalpha}
\bibliography{References}

\end{document}